\pdfoutput=1
\RequirePackage{ifpdf}
\ifpdf 
\documentclass[pdftex]{sigma}
\else
\documentclass{sigma}
\fi

\numberwithin{equation}{section}

\usepackage{enumerate}
\usepackage{empheq}
\newtheorem{Thm}{Theorem}[section]
\newtheorem{Prop}[Thm]{Proposition}
\newtheorem{Cor}[Thm]{Corollary}
\newtheorem{Lem}[Thm]{Lemma}
\newtheorem{Clm}[Thm]{Claim}

\newtheorem*{Ma}{Main Theorem}
\theoremstyle{definition}

\newtheorem{Def}[Thm]{Definition}
\newtheorem{Asu}[Thm]{Assumption}
\newtheorem{Rem}[Thm]{Remark}

\newcommand{\Ric}{\operatorname{Ric}}
\newcommand{\tr}{\operatorname{tr}}

\newcommand{\Id}{\operatorname{Id}}
\newcommand{\Vol}{\operatorname{Vol}}
\newcommand{\diam}{\operatorname{diam}}

\newcommand{\Span}{\operatorname{Span}}

\newcommand{\LIP}{\operatorname{LIP}}
\newcommand{\Lip}{\operatorname{Lip}}
\newcommand{\PCM}{\operatorname{PCM}}
\newcommand{\Sf}{\operatorname{Sf}}
\newcommand{\TestF}{\operatorname{TestF}}
\newcommand{\supp}{\operatorname{supp}}

\newcommand{\Hess}{\operatorname{Hess}}
\newcommand{\Ch}{\operatorname{Ch}}
\newcommand{\sym}{\operatorname{sym}}
\newcommand{\Fl}{\operatorname{Fl}}
\newcommand{\R}{\mathbb{R}}
\newcommand{\Z}{\mathbb{Z}}
\newcommand{\mf}{\mathfrak{m}}
\newcommand{\Div}{\operatorname{div}}
\newcommand{\RCD}{{\rm RCD}^\ast}
\newcommand{\BE}{{\rm BE}}
\newcommand{\BL}{{\rm BL}}

\newcommand{\Ha}{\mathcal{H}}
\newcommand{\nHa}{\underline{\mathcal{H}}}
\newcommand{\D}{\mathcal{D}}
\newcommand{\A}{\mathcal{A}}
\begin{document}

\newcommand{\arXivNumber}{2007.07491}

\renewcommand{\thefootnote}{}

\renewcommand{\PaperNumber}{017}

\FirstPageHeading

\ShortArticleName{Convergence to the Product of the Standard Spheres and Eigenvalues of the Laplacian}

\ArticleName{Convergence to the Product of the Standard Spheres\\ and Eigenvalues of the Laplacian\footnote{This paper is a~contribution to the Special Issue on Scalar and Ricci Curvature in honor of Misha Gromov on his 75th Birthday. The full collection is available at \href{https://www.emis.de/journals/SIGMA/Gromov.html}{https://www.emis.de/journals/SIGMA/Gromov.html}}}

\Author{Masayuki AINO}

\AuthorNameForHeading{M.~Aino}

\Address{RIKEN, Center for Advanced Intelligence Project AIP,\\ 1-4-1 Nihonbashi, Tokyo 103-0027, Japan}
\Email{\href{mailto:masayuki.aino@riken.jp}{masayuki.aino@riken.jp}}
\URLaddress{\url{https://sites.google.com/site/masayukiaino/}}

\ArticleDates{Received July 17, 2020, in final form February 07, 2021; Published online February 24, 2021}

\Abstract{We show a Gromov--Hausdorff approximation to the product of the standard spheres $S^{n-p}\times S^p$ for Riemannian manifolds with positive Ricci curvature under some pinching condition on the eigenvalues of the Laplacian acting on functions and forms.}


\Classification{53C20; 58J50}

\section{Introduction}
In this article we show that if an $n$-dimensional closed Riemannian manifold with positive Ricci curvature admits an almost parallel $p$-form ($2\leq p <n/2$) in $L^2$-sense and if the first $n+1$ eigenvalues of the Laplacian acting on functions are close to their optimal values, then the Riemannian manifold is close to the product of the standard spheres $S^{n-p}\times S^p$ with appropriate radii (Main Theorem below).
Before giving the precise statement, we provide some backgrounds.

The Lichnerowicz--Obata theorem is one of the classical theorem about the first eigenvalue of~the Laplacian.
Lichnerowicz showed the optimal comparison result for the first eigenvalue when the Riemannian manifold has positive Ricci curvature, and Obata showed that the equality of~the Lichnerowicz estimate implies that the Riemannian manifold is isometric to the standard sphere.
In the following, $\lambda_k(g)$ denotes the $k$-th positive eigenvalue of the minus Laplacian $-\Delta:=-\tr_g \Hess$ acting on functions.
\begin{Thm}[Lichnerowicz--Obata theorem]\label{LiOb}
Take an integer $n\geq 2$.
Let $(M,g)$ be an $n$-dimen\-sional closed Riemannian manifold. If $\Ric \geq (n-1) g$, then $\lambda_1(g)\geq n$.
The equality $\lambda_1(g)=n$ holds if and only if $(M,g)$ is isometric to the standard sphere of radius~$1$.
\end{Thm}

Petersen \cite{Pe1}, Aubry \cite{Au} and Honda \cite{Ho} showed the stability result of the Lichnerowicz--Obata theorem.
In the following, $d_{{\rm GH}}$ denotes the Gromov--Hausdorff distance and $S^n$ denotes the $n$-dimensional standard sphere of radius~1 (see Definition~\ref{DGH} for the definition of the Gromov--Hausdorff distance).
\begin{Thm}[\cite{Au, Ho, Pe1}]\label{PA}
For given an integer $n\geq 2$ and a positive real number $\epsilon>0$, there exists $\delta(n,\epsilon)>0$ such that if $(M,g)$ is an $n$-dimensional closed Riemannian manifold with $\Ric \geq (n-1) g$ and $\lambda_n(g)\leq n+\delta$, then $d_{{\rm GH}}\big(M,S^n(1)\big)\leq \epsilon$.
\end{Thm}
Note that Petersen considered the pinching condition on $\lambda_{n+1}(g)$, and Aubry and Honda improved it independently.

When the Riemannian manifold admits a non-trivial parallel differential form, we have the stronger estimate.
\begin{Thm}[\cite{ Ai3, gr}]\label{gros}
Let $(M,g)$ be an $n$-dimensional closed Riemannian manifold.
Assume that $\Ric\geq (n-p-1)g$ and that there exists a nontrivial parallel $p$-form on $M$ $(2\leq p\leq n/2)$.
Then, we have
$
\lambda_1(g)\geq n-p.
$
Moreover, if $p<n/2$ and $\lambda_{n-p+1}(g)=n-p$ hold, then $(M,g)$ is isometric to a product $S^{n-p}(1)\times (X,g')$, where $(X,g')$ is some $p$-dimensional closed Riemannian manifold.
\end{Thm}
To simplify the numbers appearing in the theorem, we consider the assumption $\Ric\geq $ \mbox{$(n-p-1)g$} instead of $\Ric\geq (n-1)g$.
By scaling, the estimate in Theorem~\ref{LiOb} becomes $\lambda_1(g)\geq n(n-p-1)/(n-1)$ when $\Ric\geq (n-p-1)g$.
Note that we have $n-p>n(n-p-1)/(n-1)$.

To state the almost version of Theorem~\ref{gros}, we introduce the first eigenvalue of the connection Laplacian acting on $p$-forms $\lambda_1(\Delta_{C,p})$ for a closed Riemannian manifold $(M,g)$:
\[
\lambda_1(\Delta_{C,p}):=\inf\left\{\frac{\|\nabla \omega\|_{L^2}^2}{\|\omega\|_{L^2}^2}\colon \omega\in\Gamma\Big(\bigwedge^p T^\ast M\Big) \text{ with }\omega\neq 0\right\} .
\]
Note that there exists a non zero $p$-form $\omega$ with $\|\nabla \omega\|_{L^2}^2\leq \delta\|\omega\|_{L^2}^2$ for some $\delta>0$ if and only if $\lambda_1(\Delta_{C,p})\leq \delta$ holds.
For arbitrary integers $n$, $p$ with $2\leq p\leq n/2$ and a real number $\epsilon>0$, considering a small perturbation of $S^{n-p}(1)\times S^p(r_{n,p})$, we can find an $n$-dimensional closed Riemannian manifold with $\Ric\geq (n-p-1) g$ such that $0<\lambda_1(\Delta_{C,p})<\epsilon$ holds. Here we defined $r_{n,p}:=\sqrt{(p-1)/(n-p-1)}$.
In other words, we do not have the gap theorem for the first eigenvalue of the connection Laplacian $\lambda_1(\Delta_{C,p})$ if we only assume a lower Ricci curvature bound.

Let us state the almost version of the eigenvalue estimate.
\begin{Thm}[\cite{Ai3}]
For given integers $n\geq 4$ and $2\leq p \leq n/2$, there exists a constant \mbox{$C(n,p)>0$} such that if $(M,g)$ is an $n$-dimensional closed Riemannian manifold with $\Ric_g\geq (n-p-1)g$,
then we have
\begin{gather*}
\lambda_1(g)\geq n-p-C(n,p)\lambda_1(\Delta_{C,p})^{1/2}.
\end{gather*}
\end{Thm}
This theorem recovers the estimate in Theorem~\ref{gros} when $\lambda_1(\Delta_{C,p})=0$.

We next state the approximation result to the product space.
\begin{Thm}[\cite{Ai3}]\label{prevm}
For given integers $n\geq 5$ and $2\leq p < n/2$ and a positive real number $\epsilon>0$, there exists $\delta=\delta(n,p,\epsilon)>0$ such that if $(M,g)$ is an $n$-dimensional closed Riemannian manifold with $\Ric_g\geq (n-p-1)g$,
\begin{gather*}
\lambda_{n-p+1}(g)\leq n-p+\delta
\end{gather*}
and
\begin{gather*}
\lambda_1(\Delta_{C,p})\leq \delta,
\end{gather*}
then $M$ is orientable and
\[
d_{{\rm GH}}\big(M,S^{n-p}(1)\times X\big)\leq \epsilon,
\]
where $X$ is some compact metric space.
\end{Thm}

In this article we study the structure of the metric space $X$ in this theorem and show that $X$ with some appropriate Borel measure satisfies the $\RCD(n-p-1,p)$ condition (see Proposition~\ref{strX}), which means a synthetic notion of ``$\Ric\geq n-p-1$ and $\dim\leq p$ with Riemannian structure'' (see Definition~\ref{rcd}).
As a consequence, we can show the estimate $\lambda_{n-p+2}(g)\geq p(n-p-1)/(p-1)-\epsilon$ under the assumption of Theorem~\ref{prevm} (see Theorem~\ref{n-p+2}) and the following theorem.
\begin{Ma}
For given integers $n\geq 5$ and $2\leq p < n/2$ and a positive real number $\epsilon>0$, there exists $\delta=\delta(n,p,\epsilon)>0$ such that if $(M,g)$ is an $n$-dimensional closed Riemannian manifold with $\Ric_g\geq (n-p-1)g$,
\begin{gather*}
\lambda_{n-p+1}(g)\leq n-p+\delta,\qquad
\lambda_{n+1}(g)\leq \frac{p(n-p-1)}{p-1}+\delta
\end{gather*}
and
\begin{gather*}
\lambda_1(\Delta_{C,p})\leq \delta,
\end{gather*}
then
\[
d_{{\rm GH}}\bigg(M,S^{n-p}(1)\times S^p\bigg(\sqrt{\frac{p-1}{n-p-1}}\bigg)\bigg)\leq \epsilon,
\]
and $M$ is diffeomorphic to $S^{n-p}\times S^p$.
\end{Ma}
We have the last assertion by the topological stability theorem due to Cheeger--Colding \cite[Theorem~A.1.12]{CC1}.
We show the main theorem including the case when $\lambda_1(\Delta_{C,n-p})\leq \delta$ (see Theorem~\ref{mthm}).

\section{Preliminaries}
\subsection{Basic notation}
We first recall some basic definitions and fix our convention.


Let $(M,g)$ be a closed Riemannian manifold.
For any $p\geq 1$, we use the normalized $L^p$-norm:
\begin{gather*}
\|f\|_{L^p}^p:=\frac{1}{\Vol(M)}\int_M |f|^p\,{\rm d}\mu_g,
\end{gather*}
and $\|f\|_{L^\infty}:=\mathop{\mathrm{ess~sup}}\limits_{x\in M}|f(x)|$ for a measurable function $f$ on $M$. We also use these notation for~tensors.
We have $\|f\|_{L^p}\leq \|f\|_{L^q}$ for any $p\leq q \leq \infty$.

Let $\nabla$ denote the Levi-Civita connection.
Throughout this paper,
 $0=\lambda_0(g)< \lambda_1(g) \leq \lambda_2(g) \leq\cdots \to \infty$
denotes the eigenvalues of the minus Laplacian $-\Delta=-\tr\Hess$ acting on~functions counted with multiplicities.
For $p=0,1,\ldots, n$, let
\[
\lambda_1(\Delta_{C,p}):=\inf\left\{\frac{\|\nabla \omega\|_{L^2}^2}{\|\omega\|_{L^2}^2}\colon \omega\in\Gamma\Big(\bigwedge^p T^\ast M\Big) \text{ with }\omega\neq 0\right\} .
\]

For metric space $(X,d)$ and $k\in\R_{\geq 0}$, let $\Ha^k$ denote the $k$-dimensional Hausdorff measure.
If~$0<\Ha^k(X)<\infty$, let $\nHa^k$ denote the normalized $k$-dimensional Hausdorff measure:
\[
\nHa^k:=\frac{1}{\Ha^k(X)}\Ha^k.
\]

In this article, for metric spaces $(X_i,d_i)$ ($i=1,2$), let $d_1\times d_2$ denote the distance on $X_1\times X_2$ satisfying
\[
(d_1\times d_2)^2\left((x_1,x_2),(y_1,y_2)\right)=d_1^2(x_1,y_1)+d_2^2(x_2,y_2)
\]
for $(x_1,x_2),(y_1,y_2)\in X_1\times X_2$.

\subsection{Metric measure spaces}
In this article we only consider a compact metric measure space with full support and unit total mass for simplicity of the description because it is enough for our purpose.
\begin{Def}
In this article we say that $(X,d,\mf)$ is a compact metric measure space if $(X,d)$ is a compact metric space and $\mf$ is a Borel measure with $\supp \mf =X$ and $\mf(X)=1$.
\end{Def}

We introduce some functional analytic tools on a metric measure space.
Our main references are \cite{AGS1, Gig,GP}.

\begin{Def}
Let $(X,d,\mf)$ be a compact metric measure space.
\begin{itemize}\itemsep=0pt
\item \emph{Local Lipschitz constant}. Let $\LIP(X)$ denote the set of the Lipschitz functions on $X$.
For~each $f\in \LIP(X)$ and $x\in X$, we define a local Lipschitz constant $\Lip (f)(x)$ by
\[
\Lip (f)(x):=\limsup_{y\to x} \frac{|f(x)-f(y)|}{d(x,y)}
\]
if $x\in X$ is not an isolated point, and $\Lip(f)(x)=0$ otherwise.
\item \emph{Cheeger energy}. For each $f\in L^2(X)$, we define the Cheeger energy $\Ch(f)\in[0,\infty]$ by
\[
\Ch(f):=\frac{1}{2}\inf\left\{\liminf_{i\to \infty}\int_X (\Lip(f_i))^2 \,{\rm d}\mf\colon f_i\in \LIP(X) \text{ and } \lim_{i\to\infty}\|f-f_i\|_{L^2}=0\right\} .
\]
Define
\[
W^{1,2}(X)=W^{1,2}(X,d,\mf):=\{f\in L^2(X)\colon \Ch(f)<\infty\}.
\]
We have that $W^{1,2}(X)$ is a Banach space with the norm $\|f\|_{W^{1,2}}=(\|f\|_{L^2}^2+2\Ch(f))^{1/2}$.
\item \emph{Minimal relaxed gradient}. We say that $|D f|\in L^2(X)$ is the minimal relaxed gradient of~$f\in W^{1,2}(X)$ if there exists a sequence $\{f_i\}_{i=1}^\infty$ of Lipschitz function such that
$\lim\limits_{i\to \infty}\|f-f_i\|_{L^2}=0$, $\lim\limits_{i\to\infty}\|\, |D f|-\Lip (f_i)\|_{L^2}=0$ and
\[
\Ch(f)=\frac{1}{2}\int_X |D f|^2 \,{\rm d}\mf.
\]
For any $f\in W^{1,2}(X)$, the minimal relaxed gradient $|D f|\in L^2(X)$ exists and unique. See \cite[Definition~4.2 and~Lemma 4.3]{AGS1}.
\item \emph{Sobolev-to-Lipschitz property.} We say that $(X,d,\mf)$ satisfies the Sobolev-to-Lipschitz pro\-perty if any $f\in W^{1,2}(X)$ with $|D f|\leq 1$ $\mf$-a.e.\ in $X$ is a $1$-Lipschitz function on $X$ (more precisely, $f$ has a $1$-Lipschitz representative).
\item \emph{Infinitesimally Hilbertian}. We say that $(X,d,\mf)$ is infinitesimally Hilbertian if $\Ch$ is a~quad\-ratic form.
This condition holds if and only if $\big(W^{1,2}(X),\|\cdot\|_{W^{1,2}}\big)$ is a Hilbert space.
In this case, we define $\mathcal{E}\colon W^{1,2}(X)\times W^{1,2}(X)\to \R$ by
\[
\mathcal{E}(f,g)=\frac{1}{2}(\Ch(f+g)-\Ch(f-g)).
\]
\item \emph{Laplacian}. If $(X,d,\mf)$ is infinitesimally Hilbertian, then we define
\[
\D(\Delta):=\left\{f\in W^{1,2}(X)\colon \begin{array}{@{}l@{}}\text{there exists $\Delta f\in L^2(X)$ such that we have}\\
\text{$\mathcal{E}(f,g)=-\int_X g \Delta f \,{\rm d}\mf$ for any $g\in W^{1,2}(X)$}\end{array}  \right\} .
\]
For each $f\in \D(\Delta)$, $\Delta f\in L^2(X)$ is uniquely determined.
\item \emph{The function} $\langle D f_1,D f_2\rangle$. If $(X,d,\mf)$ is infinitesimally Hilbertian, then we define \linebreak $\langle D f_1,D f_2\rangle\in L^1(X)$ for $f_1,f_2\in W^{1,2}(X)$ by
\[
\langle D f_1,D f_2\rangle=\lim_{\epsilon\to 0}\frac{1}{2\epsilon}(|D (f_1+\epsilon f_2)|^2-|D f_1|^2)\in L^1(X).
\]
This notion is well-defined by the convexity of the minimal relaxed gradient (see \cite[Definition~4.12]{AGS2}).
We have that $\langle D f_1,D f_1\rangle=|D f_1|^2$ $\mf$-a.e.\ in $X$, that $|\langle D f_1, D f_2\rangle|\leq |D f_1| |D f_2 |$ $\mf$-a.e.\ in $X$, that $\langle\cdot,\cdot\rangle$ is a symmetric bilinear form, and that
\[
\mathcal{E}(f_1,f_2)=\int_X \langle D f_1,D f_2\rangle\,{\rm d}\mf
\]
by \cite[Propositions 4.13 and 4.14, Theorem~4.18]{AGS2}.
\item \emph{Heat flow}. Let $(X,d,\mf)$ be infinitesimally Hilbertian.
Let $\{P_t f\}_{t>0}$ denote the gradient flow of the Cheeger energy $\Ch$ starting from $f\in L^2(X)$ (see \cite[Definition~5.2.5]{GP}).
The flow $\{P_t f\}_{t>0}$ is called the heat flow and characterized as the unique $C^1$ map $(0,\infty)\to L^2(X)$ (it turns out to be $C^\infty$ \cite[Proposition~5.2.12]{GP}) satisfying the following conditions (see \cite[Theorem~5.1.12]{GP}):
\begin{itemize}\itemsep=0pt
\item We have $P_t f\to f$ strongly in $L^2(X)$ as $t\to 0$.
\item For each $t>0$, we have that $P_t f\in \D(\Delta)$ and that
\[
\frac{\rm d}{{\rm d} t}P_t f =\Delta P_t f
\]
in $L^2(X)$.
\end{itemize}
Moreover, we have the following properties (see \cite[Section 5.2.2]{GP}):
\begin{itemize}\itemsep=0pt
\item For each $t>0$ and $f\in L^2(X)$, we have
\begin{gather*}
\Ch(P_t f)\leq \inf_{g\in W^{1,2}(X)}\left(\Ch(g)+\frac{\|f-g\|_{L^2}^2}{2t}\right),
\\
\|\Delta P_t f\|_{L^2}^2\leq \inf_{g\in \D(\Delta)}
\left(\|\Delta g\|_{L^2}^2+\frac{\|f-g\|_{L^2}^2}{t^2}\right).
\end{gather*}
See also \cite[Theorem~5.1.12]{GP}.
\item For each $t>0$, $P_t\colon L^2(X)\to L^2(X)$ is a linear map satisfying
\[
\int_X g P_t f\,{\rm d}\mf=\int_X f P_t g\,{\rm d}\mf
\]
for any $f,g\in L^2(X)$.
\item For each $s,t>0$, we have $P_{s+t}=P_s\circ P_t$.
\item For each $f\in \D(\Delta)$ and $s>0$, we have that
\[
\lim_{t\to 0+} \frac{P_t f-f}{t}=\Delta f
\]
in $L^2(X)$ and that $\Delta P_s f =P_s \Delta f$.
\item For each $t>0$, $c\in \R$ and $f\in L^2(X)$ with $f\leq c$ $\mf$-a.e.\ in $X$, we have $P_t f\leq c$ $\mf$-a.e.\ in $X$.
\item For each $t>0$, $p\in [1,\infty)$ and $f\in L^2(X)\cap L^p(X)$, we have $\|P_t f\|_{L^p}\leq \|f\|_{L^p}$.
In~particular, we can extend the map $P_t\colon L^2(X)\cap L^p(X)\to L^2(X)\cap L^p(X)$ to~$P_t$: $L^p(X)\to L^p(X)$.
\end{itemize}
We can also show the following properties by the above properties:
\begin{itemize}\itemsep=0pt
\item For each $f\in W^{1,2}(X)$, we have $\Ch(P_t f-f)\to 0$ as $t\to 0$. Indeed, the properties $\Ch(P_t f)\leq \Ch(f)$ and $P_t f\to f$ in $L^2$ as $t\to 0$ imply that $P_t f$ converges to $f$ weakly in~$W^{1,2}$, and so $\limsup\limits_{t\to 0}\Ch(P_t f)\leq \Ch(f)$ implies that $P_t f$ converges to $f$ strongly in~$W^{1,2}$ as~$t\to 0$.
\item For each $p\in [1,\infty)$ and $f\in L^p(X)$, we have $\|P_t f-f\|_{L^p}\to 0$ as $t\to 0$. This can be verified by applying the above properties to the truncated function $f_1=\max\{\min\{f,c\},-c\}$ for a sufficiently large constant \mbox{$c\!>\!0$} and its remaining part~\mbox{$f\!-\!f_1$}.
\end{itemize}
\item \emph{Test functions}. Let $(X,d,\mf)$ be infinitesimally Hilbertian. We define
\[
\TestF(X):=\big\{f\in\D(\Delta)\cap L^\infty(X)\colon |\nabla f|\in L^\infty(X)\text{ and }\Delta f\in W^{1,2}(X)\big\} .
\]
\item \emph{Pre-cotangent module}. We define
\[
\PCM:=\left\{\{(f_i,A_i)\}\colon\,
\begin{array}{@{}l@{}}\text{$\{A_i\}_{i=1}^\infty$ is a pairwise disjoint family of Borel subsets of $X$}\\
\text{with $\bigcup_{i} A_i=X$, $f_i\in W^{1,2}(X)$ with $\sum_{i} \int_{A_i} |D f_i|^2\,{\rm d}\mf<\infty$}
\end{array}
\right\}.
\]
We say that $\{(f_i,A_i)\}\in \PCM$ is equivalent to $\{(g_i,B_i)\}\in \PCM$ (denote it by $\{(f_i,A_i)\}\sim\{(g_i,B_i)\}$) if
\[
|D (f_i-g_j)|=0 \text{ $\mf$-a.e.\ in $A_i\cap B_j$ for each $i,j\in\Z_{>0}$}.
\]
We define
\begin{gather*}
|\cdot|\colon \PCM/{\sim} \to L^2(X),\qquad
[\{(f_i,A_i)\}]\mapsto \sum_{i=1}^\infty \chi_{A_i}|D f_i|,
\\
\|\cdot\|_{L^2}\colon \PCM/{\sim} \to [0,\infty),\quad [\{(f_i,A_i)\}]\mapsto\big\||[\{(f_i,A_i)\}]|\big\|_{L^2}\!
=\!\bigg(\!\sum_{i}\! \int_{A_i}\!\! |D f_i|^2\,{\rm d}\mf\!\bigg)^{1/2}\!\!\!,
\end{gather*}
where $\chi_{A_i}$ denotes the characteristic function.
Then, $(\PCM/{\sim},\|\cdot\|_{L^2})$ is naturally equipped with the structure of the normed vector space.
Moreover, we define
\[
\bigg(\sum_{i=1}^\infty a_i \chi_{A_i}\bigg)\cdot [\{(f_i,B_i)\}]:=[\{(a_i f_j,A_i\cap B_j)\}_{i,j}]
\]
for each $\sum_{i=1}^\infty a_i \chi_{A_i}\in \Sf(X)$ and $[\{(f_i,B_i)\}]\in \PCM/{\sim}$, where $\Sf(X)$ is defined by
\[
\Sf(X):=\left\{\sum_{i=1}^\infty a_i \chi_{A_i}\colon
\begin{array}{@{}l@{}}\text{$\{A_i\}_{i=1}^\infty$ is a pairwise disjoint family of Borel subsets of $X$}\\
\text{with $\bigcup_{i} A_i=X$, $a_i\in \R$ with $\sup_{i} |a_i|<\infty$}
\end{array} \right\} .
\]
Then, we have
$\|f\cdot \omega\|_{L^2}\leq \|f\|_{L^\infty}\|\omega\|_{L^2}$
for each $f\in\Sf(X)$ and $\omega\in \PCM/{\sim}$.
\item \emph{Cotangent module}. We define the cotangent module $L^2(T^\ast X)$ as a completion of the normed vector space $(\PCM/{\sim},\|\cdot\|_{L^2})$.
We can extend the action $\Sf(X)\times \PCM/{\sim}\to\PCM/{\sim}$ to $L^\infty(X)\times L^2(T^\ast X)\to L^2(T^\ast X)$, and
$|\cdot|\colon \PCM/{\sim} \to L^2(X)$ to $|\cdot|\colon L^2(T^\ast X) \to L^2(X)$.
Then, $L^2(T^\ast X)$ is equipped with the structure of an $L^2$-normed $L^\infty(X)$ module, i.e., we have that
\begin{gather*}
(f g)\cdot \omega=f\cdot (g\cdot \omega),\qquad
1\cdot \omega=\omega,\qquad
|\omega|\geq 0\quad \text{$\mf$-a.e.\ in X},
\\
\|\omega\|_{L^2}=\big\||\omega|\big\|_{L^2},\qquad
|f\cdot \omega|=|f||\omega|
\end{gather*}
for each $f,g\in L^\infty(X)$ and $\omega\in L^2(T^\ast X)$.
Note that we use $W^{1,2}(X)$ instead of the Sobolev class $S^2(X)$ (see \cite[Definition~2.1.4]{Gig}) in the definition of $\PCM$.
However, we can approximate elements of $S^2(X)$ by elements of $W^{1,2}(X)$ (see \cite[Proposition~2.2.5]{Gig}), and so our definition of $L^2(T^\ast X)$ coincides with \cite[Definition~2.2.1]{Gig}.
In general, $L^2(T^\ast X)$ does not need to be a Hilbert space, and $L^2(T^\ast X)$ is a Hilbert space if and only if $(X,d,\mf)$ is infinitesimally Hilbertian \cite[Proposition~2.3.17]{Gig}.
\item \emph{Pointwise scalar product}. If $(X,d,\mf)$ is infinitesimally Hilbertian, we can define a pointwise scalar product
\[
\langle[\{(f_i,A_i)\}],[\{(g_i,B_i)\}]\rangle:=\sum_{i,j=1}^\infty \chi_{A_i\cap B_j}\langle D f_i,D g_j\rangle
\]
for each $[\{(f_i,A_i)\}],[\{(g_i,B_i)\}]\in\PCM/{\sim}$ and extend it to $\langle\cdot,\cdot\rangle\colon L^2(T^\ast X)\times L^2(T^\ast X)\to L^1(X)$.
Then, $\langle\cdot,\cdot\rangle$ is symmetric and $L^2(T^\ast X)$ is a Hilbert space with the inner product defined by
\[
(\omega,\eta)\mapsto \int_X \langle\omega,\eta\rangle\,{\rm d}\mf
\]
for each $\omega,\eta\in L^2(T^\ast X)$.
Clearly, for each $\omega,\eta\in L^2(T^\ast X)$, we have $\langle \omega,\omega\rangle=|\omega|^2$ and $|\langle\omega,\eta\rangle|\leq |\omega||\eta|$ $\mf$-a.e.\ in $X$.
\item \emph{Differential}. We define the differential ${\rm d}\colon W^{1,2}(X)\to L^2(T^\ast X)$ by ${\rm d} f=[(f,X)]$ for each $f\in W^{1,2}(X)$. Clearly, we have that $|D f|=|{\rm d} f|\in L^2(X)$ for each $f\in W^{1,2}(X)$.
Moreover, if $(X,d,\mf)$ is infinitesimally Hilbertian, then we have $\langle D f,D g\rangle=\langle {\rm d} f , {\rm d} g\rangle\in L^1(X)$ for each $f,g\in W^{1,2}(X)$.
\item \emph{Tangent module}. We define the tangent module $L^2(T X)$ by
\[
L^2(T X):=\left\{V\colon \begin{array}{@{}l@{}}\text{$V\colon L^2(T^\ast X)\to L^1(X)$ is a bounded linear operator such that}\\
\text{$V(f\cdot \omega)=f V(\omega)$ holds for all $f\in L^\infty(X)$ and $\omega\in L^2(T^\ast X)$}
\end{array} \right\} .
\]
If $(X,d,\mf)$ is infinitesimally Hilbertian, then the map
\[
L^2(T^\ast X)\to L^2(T X),\qquad\omega\mapsto \langle\omega,\cdot\rangle
\]
is bijective (see \cite[Theorem~1.2.24]{Gig}).
Under this identification, $L^2(T X)$ is equipped with the structure of an $L^2$-normed $L^\infty(M)$ module and a pointwise scalar product $\langle\cdot,\cdot\rangle\colon L^2(T X)\times L^2(T X)\to L^1(X)$.
Note that even if $(X,d,\mf)$ is not infinitesimally Hilbertian, the tangent module $L^2(T X)$ is naturally equipped with the structure of an~$L^2$-normed $L^\infty(M)$ module (see \cite[Definition~1.2.6 and~Proposition~1.2.14]{Gig}).
\item \emph{Gradient}. If $(X,d,\mf)$ is infinitesimally Hilbertian, we define a map $\nabla\colon W^{1,2}(X)\to L^2(T X)$ by $\nabla f=\langle {\rm d} f,\cdot\rangle$ for each $f\in W^{1,2}(X)$. Clearly, we have that $|\nabla f|=|D f|=|{\rm d} f|\in L^2(X)$ and $\langle\nabla f,\nabla g\rangle=\langle D f, D g\rangle\in L^1(X)$ for each $f,g \in W^{1,2}(X)$.
\item \emph{Divergence}. If $(X,d,\mf)$ is infinitesimally Hilbertian, then we define
\[
\D(\Div):=\left\{V\in L^{2}(T X)\colon \begin{array}{@{}l@{}}\text{there exists $\Div V \in L^2(X)$ such that we have}\\
\text{$\int_X \langle V, \nabla g\rangle\, {\rm d}\mf=-\int_X g\, \Div V \,{\rm d}\mf$ for any $g\in W^{1,2}(X)$}\end{array}  \right\} .
\]
For each $V\in\D(\Div)$, $\Div V\in L^2(X)$ is uniquely determined.
\item \emph{Symmetric part of the covariant derivative} \cite[Definition~5.4]{AT2}. Let $(X,d,\mf)$ be infinitesimally Hilbertian.
For a vector field $V\in \D(\Div)$ we write $D^{\sym} V\in L^2(X)$ if there exists $c>0$ such that we have
\begin{gather}\label{symd}
\left|\int_X \langle V,\nabla f\rangle \Delta g+ \langle V,\nabla g\rangle\Delta f -\langle\nabla f,\nabla g\rangle\Div V\,{\rm d}\mf\right|\leq c \|\nabla f\|_{L^4}\|\nabla g\|_{L^4}
\end{gather}
for any $f,g\in \D(\Delta)$ with $|\nabla f|,|\nabla g|\in L^4(X)$ and $\Delta f,\Delta g\in L^4(X)$.
\end{itemize}

\end{Def}
\subsection[The RCD* condition and some properties]{The RCD$^{\boldsymbol *}$ condition and some properties}
In this subsection we recall the definition of the $\RCD(K,N)$ space and its properties.
\begin{Def}
We say that an infinitesimally Hilbertian metric measure space $(X,d,\mf)$ satisfies the Bakry--\'{E}mery condition $\BE(K,N)$ with $K\in\R$ and $N\in [1,\infty)$ if for all $u\in\D(\Delta)$ with $\Delta u\in W^{1,2}(X)$ and all $\phi\in\D(\Delta)\cap L^\infty(X)$ with $\phi\geq0$ and $\Delta \phi\in L^\infty(X)$, we have
\begin{gather*}
\frac{1}{2}\int_{X}\Delta \phi|\nabla u|^2 \,{\rm d}\mf\geq \int_{X} \phi\bigg(\langle\nabla \Delta u,\nabla u\rangle+K|\nabla u|^2+\frac{1}{N}(\Delta u)^2\bigg) \,{\rm d}\mf.
\end{gather*}
\end{Def}

\begin{Def}
We say that an infinitesimally Hilbertian metric measure space $(X,d,\mf)$ satis\-fies the Bakry--Ledoux condition $\BL(K,N)$ with $K\in\R$ and $N\in [1,\infty)$ if for all $u\in W^{1,2}(X)$ and $t>0$ we have
\[
|\nabla P_t f|^2+\frac{2t C(t)}{N}|\Delta P_t f|^2\leq {\rm e}^{-2Kt}P_t\big(|\nabla f|^2\big)
\]
$\mf$-a.e.\ in X, where $C(t)>0$ is a function satisfying $C(t)=1+O(t)$ as $t\to 0$.
\end{Def}
\begin{Thm}[{\cite[Theorem~4.8 and Proposition~4.9]{EKS}}]
An infinitesimally Hilbertian metric measure space $(X,d,m)$ satisfies the $\BE(K,N)$ condition if and only if it satisfies the $\BL(K,N)$ condition.
\end{Thm}
Let us recall an equivalent version of the definition of the $\RCD(K,N)$ condition (see \cite[Theorem~7 and~Definition~3.16]{EKS} for the equivalence).
If the total mass is finite, the $\RCD(K,N)$ condition is equivalent to the $\mathrm{RCD}(K,N)$ condition by \cite[Corollary~13.7]{CM}.

\begin{Def}\label{rcd}
We say that a compact infinitesimally Hilbertian metric measure space $(X,d,\mf)$ satisfies the $\RCD(K,N)$ condition with $K\in\R$ and $N\in [1,\infty)$ if $(X,d,\mf)$ satisfies the $\BE(K,N)$ condition and the Sobolev-to-Lipschitz property.
\end{Def}
For more general metric measure space, we add the volume growth assumption to Definition~\ref{rcd}. However, it is automatically satisfied in our situation because we assume that $\mf(X)=1$.
Note that Definition~\ref{rcd} implies that $(X,d)$ is a geodesic space by \cite[Theorems~3.9 and~3.10]{AGS3} and \cite[Theorem~2.5.23]{BBI}.
The original definition also implies this property (see \cite[Remark 3.8]{EKS} and \cite[Remark~4.6]{Stu1}).

The definition of the $\RCD(K,N)$ condition is consistent to the smooth case.
\begin{Prop}[{\cite[Proposition~4.21]{EKS}}]
For any $n$-dimensional closed Riemannian mani\-fold $(M,g)$ and real numbers $K\in \R$ and $N\in[1,\infty)$, we have that $(M,d_g, \nHa^n)$ satisfies the $\RCD(K,N)$ condition if and only if $\Ric\geq K g $ and $n\leq N$ hold.
\end{Prop}

Let us compare the local Lipschitz constant and the minimal relaxed gradient for Lipschitz functions.
If the $\RCD(K,N)$ condition holds, then we have the doubling condition \cite[Corollary~2.4]{Stu} and the weak Poincar\'{e} inequality \cite[Theorem~1.1]{Raj}.
Moreover, our minimal relaxed gradient coincides with Cheeger's minimal generalized upper gradient \cite[Definition~2.9]{Ch0} by \cite[Theorem~6.2]{AGS1}.
Thus, we have the following theorem by \cite[Theorem~6.1]{Ch0}:
\begin{Thm}
Let $(X,d,\mf)$ be a compact metric measure space satisfying the $\RCD(K,N)$ condition.
Then, for any $f\in \LIP(X)$, we have $\Lip f=|\nabla f|$ $\mf$-a.e.\ in $X$.
\end{Thm}

Under the $\RCD(K,N)$ condition, we have that
\[
\TestF(X)=\left\{f\in\D(\Delta)\cap \LIP(X)\colon \Delta f\in W^{1,2}(X)\right\} .
\]

Let us make a remark on the heat kernel.
Let $(X,d,\mf)$ be a compact metric measure space satisfying the $\RCD(K,N)$ condition.
Then, $\mathcal{E}$ is a strongly local Dirichlet form on $(X,\mf)$ by \cite[Proposition~4.8]{AGS1} and \cite[Proposition~4.11]{AGS2}, and we have
\[
d(x,y)=\sup\{|f(x)-f(y)|\colon\text{$f\in \LIP(X)$ with $|\nabla f|\leq 1$ $\mf$-a.e.\ in $X$}\}
\]
by \cite[Theorem~3.9]{AGS3}.
Since we have the doubling condition by \cite[Corollary~2.4]{Stu} and the strong local $(2,2)$ Poincar\'{e} inequality \cite[Property (Ic)]{Stu3} by \cite[Theorem~1.1]{Raj} and \cite[Theorem~1]{HK}, 
we can apply \cite[Proposition~2.3]{Stu2} and \cite[Proposition~3.1]{Stu3} (see also \cite[Theorem~3.5]{Stu3}), and so there exists a locally H\"{o}lder continuous function $p\colon (0,\infty)\times X\times X\to\R$ such that
\[
P_t f (x)=\int_X p(t,x,y) f(y)\,{\rm d}\mf(y)
\]
holds for any $f\in L^1(X)$.
By \cite[Theorem~1.2]{JLZ}, for any $\epsilon>0$, there exist constants $C_i=C_i(\epsilon,K,N)>1$ such that
\[
\frac{C_1^{-1}}{\mf(B_{\sqrt{t}}(x))}\exp\bigg(\!{-}\frac{d^2(x,y)}{(4-\epsilon)t}-C_2 t\bigg)\leq
p(t,x,y)\leq \frac{C_1}{\mf(B_{\sqrt{t}}(x))}\exp\bigg(\!{-}\frac{d^2(x,y)}{(4+\epsilon)t}+C_2 t\bigg)
\]
holds for each $x,y\in X$ and $t>0$.
Here, we defined $B_r(x):=\{z\in X\colon d(x,z)<r\}$ for $x\in X$ and $r>0$.
By this and the Bishop--Gromov inequality \cite[Theorem~2.3]{Stu},
we have the following:
\begin{itemize}\itemsep=0pt
\item For any $f\in L^1(X)$, we have
\begin{gather}\label{linfl1}
\|P_t f\|_{L^\infty}\leq C(t)\|f\|_{L^1}.
\end{gather}
\item For any $f\in C(X)$, the function
\[
[0,\infty)\times X\to \R,\qquad (t,x)\mapsto (P_t f)(x)
\]
is continuous. Here, we defined $P_0 f:=f$.
\end{itemize}
For any $f\in W^{1,2}(X)$ and $t>0$, we have $P_t f\in \LIP(X)$ by (\ref{linfl1}) for $|\nabla f|^2\in L^1(X)$, the $\BL(K,N)$ condition and the Sobolev-to-Lipschitz property, and so $P_t f\in \TestF(X)$ by $\Delta P_t f=P_{t/2}\Delta P_{t/2} f\in W^{1,2}(X)$.
In particular, $\TestF(X)\subset W^{1,2}(X)$ is dense.
As a corollary, we have $P_t f=P_{t/2}P_{t/2}f\in \TestF(X)$ for any $f\in L^2(X)$ and $t>0$.
Note that since we assumed that $(X,d,\mf)$ is compact and $\mf(X)=1$, we can skip the truncation procedure.

{\samepage
We next recall some basic facts about the spectrum of $-\Delta$ on a compact metric measure space $(X,d,\mf)$ satisfying the $\RCD(K,N)$ condition.
By \cite[Theorem~1.1]{Raj}, \cite[Theorem~1]{HK} and \cite[Theorem~8.1]{HK2}, the inclusion $W^{1,2}(X)\to L^2(X)$ is a compact operator.
Thus, the spectrum of~$-\Delta$ is discrete and positive:
\[
0=\lambda_0<\lambda_1\leq \lambda_2\leq\cdots\to \infty
\]
as the smooth case.
See also the proof of \cite[Theorem~4.22]{EKS}.
Let $\{\phi_i\}_{i=0}^\infty$ be the corresponding eigenfunctions.
Then,
\[
\bigoplus_{i=0}^\infty \R\phi_i=\bigg\{\sum_{i=0}^k a_i \phi_i\colon k\in\Z_{\geq 0}\text{ and }a_i\in \R\ (i=0,\ldots,k)\bigg\}
\]
is dense in $L^2(X)$.

}

Finally, let us recall the notion of the regular Lagrangian flow, which is a flow for a vector field in a non-smooth setting.
Although time dependent vector fields are considered in \cite{AT1}, we only deal with time independent vector fields because it is enough for our purpose.
\begin{Def}[\cite{AT1}]
Let $(X,d,\mf)$ be a compact measure space satisfying the $\RCD(K,N)$ condition and take $T>0$.
We say that $\Fl^V \colon [0,T]\times X\to X$ is a regular Lagrangian flow for a~vector field $V\in L^2(T X)$ if
the following properties hold:
\begin{itemize}\itemsep=0pt
\item[(i)] There exists a constant $C>0$ such that
\[
\big(\Fl^V_s\big)_\ast \mf\leq C \mf
\]
holds for any $s\in [0,T]$.
\item[(ii)] For each $x\in X$, the curve
\[
[0,T]\to X,\qquad
s\mapsto \Fl_s^V(x)
\]
is continuous and $\Fl_0^V=\Id_X$.
Moreover, the map $X\to C([0,T];X),\, x\mapsto \Fl^V(\cdot,x)$ is Borel measurable.
Here, $C([0,T];X)$ is equipped with the topology induced by the uniformly convergence.
\item[(iii)] For any Lipschitz function $f$ on $X$,
we have that
$f\big(\Fl^V(\cdot ,x)\big)\in W^{1,1}(0,T)$ for $\mf$-a.e.\ $x\in X$ and
\[
\frac{\rm d}{{\rm d} t}f\big((\Fl^V(t,x)\big)={\rm d} f(V) \big((\Fl^V(t,x)\big)
\]
for $\mathcal{L}^1\times\mf $-a.e.\ $(t,x)\in (0,T)\times X$.
\end{itemize}
\end{Def}
Note that the Borel measurability in (ii) can be verified under the assumption of the following existence theorem because we construct the flow using the disintegration theorem in \cite[Theorem~8.4]{AT1} and \cite[Theorem~7.8]{AT2}.

We use the following form of the result of \cite{AT1}. See also \cite{AT2}.
\begin{Thm}[Ambrosio--Trevisan \cite{AT1}]\label{RLF}
Let $(X,d,\mf)$ be a metric measure space satisfying the $\RCD(K,N)$ condition and take $T>0$.
For any vector field $V\in\mathcal{D}(\Div)$ with $D^{\sym}V\in L^2(X)$ and $(\Div V)^{-}\in L^\infty (X)$ $((\Div V)^{-}$ denotes the negative part of the divergence $\Div V)$, a regular Lagrangian flow $\Fl^V \colon [0,T]\times X\to X$ exists and unique, in the sense that if $\widetilde{\Fl}^V$ is another flow, then for $\mf$-a.e.\ $x\in X$ we have that $\Fl^V(s,x)=\widetilde{\Fl}^V(s,x)$ for every $s\in[0,T]$.
Moreover, we have that
\begin{gather}\label{vol}
\big(\Fl^V_s\big)_\ast\mf\leq \exp(T\|(\Div V)^-\|_{L^\infty})\mf
\end{gather}
for all $s\in [0,T]$.
\end{Thm}
\begin{Rem}
Let us give some comments about which assertions in \cite{AT1} correspond to Theorem~\ref{RLF}.
We set $\A:=\LIP(X)$.
The concept of the regular Lagrangian flow is closely related to the continuity equation
\[
\frac{\rm d}{{\rm d}t} u_t+\Div(u_t V)=0.
\]
For $\overline{u}\in L^\infty(X)$ with $\overline{u}\geq 0$, there exists a weakly continuous weak solution (in the duality with~$\A$) $u\in L^\infty([0,T]; L^\infty(X))$ of the continuity equation with initial condition $\overline{u}$ in the sense of \cite[Definition~4.2]{AT1} satisfying
\[
\|u_t\|_{L^\infty}\leq \|\overline{u}\|_{L^\infty}\exp(T\|(\Div V)^{-}\|_{L^\infty})
\]
and $u_t\geq 0$ for each $t\in [0,T]$ by \cite[Theorems~4.3 and~4.6]{AT1} (see also \cite[Theorem~6.1]{AT2}).
Note that \cite[Theorem~4.6]{AT1} deals with approximated solutions.
However, since we get the solution of the continuity equation as a weak* limit of them, we have the same estimate.
Moreover, since we have the $L^4-\Gamma$ inequality \cite[Definition~5.1 and~Corollary~6.3]{AT1}, the solution is unique by \cite[Theorem~5.4]{AT1} (see also \cite[Theorem~6.4]{AT2}) putting $p=s=r=4$ and $q=2$.
Thus, by \cite[Theorem~9.2]{AT1}, we can apply \cite[Theorem~7.7]{AT2} (see also \cite[Theorem~8.3]{AT1}), and so there exists a~unique regular Lagrangian flow for $V$.
Moreover, its proof shows the estimate (\ref{vol}).
\end{Rem}

\subsection{Gromov--Hausdorff convergence and functions}
In this subsection we recall some properties about the Gromov--Hausdorff convergence.
\begin{Def}[Hausdorff distance]\label{Dhau}
Let $(X,d)$ be a metric space.
For each point $x_0\in X$, subsets $A,B\subset X$ and $r>0$, define
\begin{gather*}
d(x_0,A):=\inf\{d(x_0,a)\colon a\in A\},
\\
B_{r}(A):=\{x\in X\colon d(x,A)<r\},
\\
d_{{\rm H},d}(A,B):=\inf\{\epsilon>0\colon A\subset B_{\epsilon}(B) \text{ and } B\subset B_{\epsilon}(A)\}.
\end{gather*}
We call $d_{{\rm H},d}$ the Hausdorff distance.
\end{Def}
The Hausdorff distance defines a metric on the collection of compact subsets of $X$.
\begin{Def}[Gromov--Hausdorff distance]\label{DGH}
Let $(X,d_X)$, $(Y,d_Y)$ be metric spaces.
Define
\[
d_{{\rm GH}}(X,Y):=\inf\left\{d_{{\rm H},d}(X,Y)\colon \begin{array}{@{}l@{}}\text{$d$ is a metric on $X\coprod Y$ such that}\\
\text{$d|_X=d_X$ and $d|_Y=d_Y$}
\end{array}\right\} .
\]
\end{Def}
The Gromov--Hausdorff distance defines a metric on the set of isometry classes of compact metric spaces (see \cite[Proposition~11.1.3]{Pe3}).

\begin{Def}[$\epsilon$-Hausdorff approximation map]\label{hap}
Let $(X,d_X)$, $(Y,d_Y)$ be metric spaces.
We~say that a map $\psi\colon X\to Y$ is an $\epsilon$-Hausdorff approximation map for $\epsilon>0$ if the following two conditions hold.
\begin{itemize}\itemsep=0pt
\item[(i)] For all $a,b\in X$, we have $|d_X(a,b)-d_Y(\psi(a),\psi(b))|< \epsilon$,
\item[(ii)] $\psi(X)$ is $\epsilon$-dense in $Y$, i.e., for all $y\in Y$, there exists $x\in X$ with $d_Y(\psi(x),y)< \epsilon$.
\end{itemize}
\end{Def}
If there exists an $\epsilon$-Hausdorff approximation map from $X$ to $Y$, then we can show that $d_{{\rm GH}}(X,Y)\leq 3\epsilon/2$.
Conversely, if $d_{{\rm GH}}(X,Y)< \epsilon$, then there exists a $2\epsilon$-Hausdorff approximation map from $X$ to $Y$.
\begin{Def}
Suppose that a sequence of $n$-dimensional closed Riemannian manifolds $\{(M_i,g_i)\}_{i\in \Z_{>0}}$ with $\Ric_i\geq K g_i$ and $\diam (M_i)\leq D$ ($K\in \R$, $D>0$) converges to a metric space $(X,d)$ in the Gromov--Hausdorff topology.
Fix a sequence of $\epsilon_i$-Hausdorff approximation maps $\psi_i\colon M_i\to X$, where $\{\epsilon_i\}$ is some sequence of positive real numbers with $\lim\limits_{i\to \infty}\epsilon_i=0$.
\begin{itemize}\itemsep=0pt
\item We say a sequence $\{x_i\}$ with $x_i\in M_i$ converges to $x\in X$ if $\lim\limits_{i\to \infty}\psi_i(x_i)= x$ in~$X$ $\big($denote it by $x_i\stackrel{{\rm GH}}{\to} x\big)$.
\item Let $\mf$ be a Borel measure on $X$.
We say that a sequence $\{(M_i,g_i,\nHa^n)\}$ converges to a~metric measure space $(X,d,\mf)$ in the measured Gromov--Hausdorff topology if
\[
\lim_{i\to \infty}\nHa^n(B_r(x_i))=\mf (B_r(x))
\]
holds for any $r>0$, $x_i\in M_i$ and $x\in X$ with $x_i \stackrel{{\rm GH}}{\to} x$.
Note that taking a subsequence, such a limit measure exists by \cite[Theorems~1.6 and~1.10]{CC1}.
Moreover, $(X,d,\mf)$ satisfies the $\RCD(K,N)$ condition by \cite[Theorem~3.22]{EKS}.
\end{itemize}
Suppose that a sequence $\{(M_i,g_i,\nHa^n)\}$ converges to $(X,d,\mf)$ in the measured Gromov--Haus\-dorff topology.
\begin{itemize}\itemsep=0pt
\item We say that $f_i\in L^2(M_i)$ ($i\in\Z_{>0}$) converges to $f\in L^2(X)$ strongly at $x\in X$ \cite[Defi\-nition~3.7]{Ho1} if we have that
\begin{gather*}
\lim_{r\to 0}\limsup_{i\to \infty}\bigg(\frac{1}{\nHa^n(B_r (x_i))}\int_{B_r(x_i)}\bigg|f_i-\frac{1}{\mf(B_r(x))}\int_{B_r(x)}f\,{\rm d}\mf\bigg|\,{\rm d}\nHa^n\bigg)=0,
\\
\lim_{r\to 0}\limsup_{i\to \infty}\bigg(\frac{1}{\mf(B_r (x))}\int_{B_r(x)}\bigg|f-\frac{1}{\nHa^n(B_r(x_i))}\int_{B_r(x_i)}f_i\,{\rm d}\nHa^n\bigg|\,{\rm d}\mf\bigg)=0
\end{gather*}
for all $x_i\in M_i$ with $x_i \stackrel{{\rm GH}}{\to} x$.
\item We say that $f_i\in L^2(M_i)$ ($i\in\Z_{>0}$) converges to $f\in L^2(X)$ weakly in $L^2$ ($L^2$ boundedness and weakly convergence \cite[Definition~3.4 and~Proposition~3.17]{Ho1}) if
\begin{gather*}
\sup_{i}\|f_i\|_{L^2}<\infty,
\end{gather*}
and for all $r>0$, $x_i\in M_i$ and $x\in X$ with $x_i \stackrel{{\rm GH}}{\to} x$, we have
\begin{gather*}
\lim_{i\to \infty}\int_{B_r(x_i)} f_i \,{\rm d}\nHa^n=\int_{B_r(x)} f \,{\rm d}\mf.
\end{gather*}
Note that we have $\|f\|_{L^2}\leq \liminf\limits_{i\to\infty}\|f_i\|_{L^2}$ by \cite[Proposition~3.29]{Ho1}.
\item We say that $f_i\in L^2(M_i)$ ($i\in\Z_{>0}$) converges to $f\in L^2(X)$ strongly in $L^2$ \cite[Definition~3.21 and Proposition~3.31]{Ho1} if $f_i$ converges to $f$ weakly in $L^2$, and
\begin{gather*}
\limsup_{i\to \infty} \|f_i\|_{L^2}\leq \|f\|_{L^2}
\end{gather*}
holds.
\item We say that $V_i\in L^2(T M_i)$ ($i\in\Z_{>0}$) converges to $V\in L^2(T X)$ weakly in $L^2$ \cite[Definition~3.42]{Ho1} if
\begin{gather*}
\sup_{i}\|V_i\|_{L^2}<\infty,
\end{gather*}
and for all $r>0$, $y_i,z_i\in X_i$ and $y,z\in X$ with $y_i\stackrel{{\rm GH}}{\to} y$, $z_i\stackrel{{\rm GH}}\to z$, we have
\begin{gather*}
\lim_{i\to \infty}\int_{B_r(y_i)} \langle V_i, \nabla r_{z_i}\rangle \,{\rm d}\nHa^n=\int_{B_r(y)} \langle V, \nabla r_{z}\rangle \,{\rm d}\mf,
\end{gather*}
where $r_z(x):=d(z,x)$ for each $x\in X$.
Note that we have $\|V\|_{L^2}\leq \liminf\limits_{i\to \infty}\|V_i\|_{L^2}$ by \cite[Proposition~3.64]{Ho1}.
\item We say that $V_i\in L^2(T M_i)$ ($i\in\Z_{>0}$) converges to $X\in L^2(T X)$ strongly in $L^2$ \cite[Definition~3.58 and~Proposition~3.66]{Ho1} if $V_i$ converges to $V$ weakly in $L^2$ and
\begin{gather*}
\limsup_{i\to \infty} \|V_i\|_2\leq \|V\|_2
\end{gather*}
holds.
\end{itemize}

\end{Def}
\begin{Prop}[{\cite[Proposition~3.32]{Ho1}}]\label{eqconv}
Suppose that a sequence of $n$-dimensional closed Riemannian manifolds $\{(M_i,g_i,\nHa^n)\}_{i\in \Z_{>0}}$ with $\Ric_i\geq K g_i$ and $\diam (M_i)\leq D$ $(K\in \R$, \mbox{$D>0)$} converges to a compact metric measure space $(X,d,\mf)$ in the measured Gromov--Hausdorff topo\-logy.
For any $f_i\in L^\infty(M_i)$ $(i\in\Z_{>0})$ with $\sup\|f_i\|_{L^\infty}<\infty$ and $f\in L^\infty (X)$, the following conditions are mutually equivalent:
\begin{itemize}\itemsep=0pt
\item[$(i)$] $f_i\to f$ strongly at a.e.\ $x\in X$.
\item[$(ii)$] $f_i\to f$ strongly in $L^2$.
\end{itemize}
\end{Prop}
Note that the implication $(i)\Rightarrow (ii)$ is a direct consequence of \cite[Definitions~3.21 and~3.25, Proposition~3.24]{Ho1}.

\begin{Thm}[{\cite[Theorem~1.3]{Ho1}}]\label{convf}
Suppose that a sequence of $n$-dimensional closed Riemannian manifolds $\{(M_i,g_i,\nHa^n)\}_{i\in \Z_{>0}}$ with $\Ric_i\geq K g_i$ and $\diam (M_i)\leq D$ $(K\in \R$, $D>0)$ converges to a compact metric measure space $(X,d,\mf)$ in the measured Gromov--Hausdorff topo\-logy.
If $f_i\in L^2(M_i)\cap C^2(M_i)$ $(i\in\Z_{>0})$ converges to $f\in L^2(X)$ weakly in $L^2$, and satisfies
\[
\sup_{i\in\Z_{>0}} \left(\|f_i\|_{W^{1,2}}+\|\Delta f_i\|_{L^2}\right)<\infty,
\]
then we have the following:
\begin{itemize}\itemsep=0pt
\item[$(i)$] $f_i \to f$ and $\nabla f_i\to \nabla f$ strongly in $L^2$,
\item[$(ii)$] $f\in \D(\Delta_X)$ and $\Delta f_i\to \Delta f$ weakly in $L^2$.
\end{itemize}
\end{Thm}

\subsection[Convergence to the product space S{n-p}(1) x X]{Convergence to the product space $\boldsymbol{S^{n-p}(1)\times X}$}
We summarize several results proven in \cite{Ai3} for later use.

\begin{Prop}[{\cite[Proposition~4.17, Lemma 4.22 and Theorem~4.47]{Ai3}}]\label{prev}
For given integers $n\geq 5$ and $2\leq p < n/2$ and a positive real number $\epsilon>0$, there exists $\delta=\delta(n,p,\epsilon)>0$ such that the following properties hold.
Let $(M,g)$ be an $n$-dimensional closed Riemannian manifold with $\Ric_g\geq (n-p-1)g$.
Assume that $\lambda_{n-p+1}(g)\leq n-p+\delta$ and that either
$\lambda_1(\Delta_{C,p})\leq \delta$ or~$\lambda_1(\Delta_{C,n-p})\leq \delta$.
Then, for any $f\in\Span_{\R}\{f_1,\ldots,f_{n-p+1}\}$ with $\|f\|_2^2=1/(n-p+1)$, we~have the following:
\begin{itemize}\itemsep=0pt
\item[$(i)$] There exists a measurable subset $V_f\subset M$ such that
$\Vol(M\setminus V_f)\leq \epsilon \Vol(M)$ and $|f^2+|\nabla f|^2-1|\leq \epsilon$ holds in $V_f$.
\item[$(ii)$] There exists a non-empty compact subset $A_f\subset M$ such that $|f(x)-1|\leq\epsilon$ for any $x\in A_f$, $|f(x)-\cos d(x,A_f)|\leq \epsilon$ for any $x\in M$ and $\sup_{x\in M} d(x,A_f)\leq \pi+\epsilon$ hold.
\item[$(iii)$] Define $\widetilde{\Psi}\colon M\to \R^{n-p+1}$ by $\widetilde{\Psi}(x)=(f_1(x),\ldots,f_{n-p+1}(x))$ $(x\in M)$.
Then, we have $||\widetilde{\Psi}(x)|-1|\leq \epsilon$ for any $x\in M$.
\item[$(iv)$] Choose $a_f(x)\in A_f$ such that
$d(x,A_f)=d(x,a_f(x))$ for each $x\in M$. Then, we have that the map
\[
\Phi_f \colon\ M\to S^{n-p}(1)\times A_f,\qquad
x\mapsto \bigg(\frac{\widetilde{\Psi}(x)}{|\widetilde{\Psi}(x)|}, a_f(x)\bigg)
\]
is an $\epsilon$-Hausdorff approximation.
\end{itemize}
\end{Prop}

\begin{Prop}[{\cite[Corollary~4.53]{Ai3}}]\label{impr}
For given integers $n\geq 5$ and $2\leq p < n/2$ and a~positive real number $\epsilon>0$, there exists $\delta=\delta(n,p,\epsilon)>0$ such that if $(M,g)$ is an $n$-dimensional closed Riemannian manifold with $\Ric_g\geq (n-p-1)g$,
$\lambda_{n-p}(g)\leq n-p+\delta$
and
$\lambda_1(\Delta_{C,n-p})\leq \delta,$
then we have
$
\lambda_{n-p+1}(g)\leq n-p+\epsilon.
$
\end{Prop}
If we assume $\lambda_1(\Delta_{C,p})\leq \delta$ instead of $\lambda_1(\Delta_{C,n-p})\leq \delta$ in Proposition~\ref{impr}, the assertion fails (see \cite[Corollary~3.2 and~Proposition~3.3]{Ai3}).

\section{Structure of the limit}
\subsection{Splitting of the measure}
In this subsection we show that there exists a Borel measure $\mf_X$ on X such that $\mf=\nHa^{n-p}\times \mf_X$ holds under Assumption~\ref{convX} below.

\begin{Asu}\label{convX}
Take $n\geq 5$ and $2\leq p < n/2$.
Let $\{(M_i,g_i)\}_{i\in\mathbb{N}}$ be a sequence of $n$-dimen\-sional closed Riemannian manifolds with $\Ric_{g_i}\geq (n-p-1)g_i$ that satisfies $\lim\limits_{i\to\infty}\lambda_{n-p+1}(g_i)=n-p$ and either $\lim\limits_{i\to \infty}\lambda_1(\Delta_{C,p},g_i)=0$ or $\lim\limits_{i\to \infty}\lambda_1(\Delta_{C,n-p},g_i)=0$.
Let $f_{1,i},\ldots,f_{n-p+1,i}$ denote the first $n-p+1$ eigenfunctions on $(M_i,g_i)$ with $\|f_{k,i}\|_{L^2}^2=1/(n-p+1)$ ($k=1,\ldots,n-p+1$).
Put
\[
\Psi_i\colon\  M_i\to S^{n-p}(1),\qquad
x\mapsto \frac{(f_{1,i}(x),\ldots,f_{n-p+1,i}(x))}{|(f_{1,i}(x),\ldots,f_{n-p+1,i}(x))|}.
\]
Let $X$ be a compact metric space and $\mf$ be a Borel measure on $S^{n-p}(1)\times X$
with unit volume.
Suppose that, for each $i$, there exists a map $b_i\colon M_i\to X$ such that
the map
\[
(\Psi_i,b_i)\colon\ M_i\to S^{n-p}(1)\times X
\]
is an $\epsilon_i$-approximation, where $\{\epsilon_i\}$ is some sequence of positive real numbers with $\epsilon_i\to 0$ as $i\to \infty$.
Suppose that the sequence $\{(M_i,g_i, \nHa^n)\}$ converges to $(S^{n-p}(1)\times X,\mf)$ in the measured Gromov--Hausdorff topology.
Put $M:=S^{n-p}(1)\times X$.
Let
$p_1\colon M\to S^{n-p}(1)$ and $p_2\colon M\to X$ be the projections.
Define
$f_u\colon M\to\R$ by $f_u(x):=p_1(x)\cdot u$ for each $u\in S^{n-p}(1)\subset \R^{n-p+1}$.
Note that $(M,\mf)$ satisfies the $\RCD(n-p-1,n)$ condition by \cite[Theorem~3.22]{EKS}.
\end{Asu}

The goal of this section is to prove the following proposition.
\begin{Prop}\label{strX}
Under Assumption~{\rm \ref{convX}}, there exists a Borel measure $\mf_X$ on X such that $\mf=\nHa^{n-p}\times \mf_X$ holds and $(X,\mf_X)$ satisfies the $\mathrm{RCD}(n-p-1,p)$ condition.
\end{Prop}
In this subsection, we show the splitting of the measure $\mf=\nHa^{n-p}\times \mf_X$.
Our approach has been inspired by~\cite{GR2018}.
We first show the following easy lemma.
\begin{Lem}\label{hfun}
Define
\[
f_k\colon\ S^{n-p}(1)\times X\to \R,\, ((u_1,\ldots,u_{n-p+1}),x)\mapsto u_k
\]
for each $k=1,\ldots, n-p+1$.
Then, we have that $f_k \in\TestF(M)$ and the following properties:
\begin{itemize}\itemsep=0pt
\item[$(i)$] For each $k=1,\ldots,n-p+1$, the sequence $\{f_{k,i}\}$ converges to $f_k$ strongly in $L^2$ as $i\to\infty$.
\item[$(ii)$] For each $k=1,\ldots,n-p+1$, the sequence $\{\nabla f_{k,i}\}$ converges to $\nabla f_k$ strongly in $L^2$ as $i\to\infty$.
\item[$(iii)$] For each $k=1,\ldots,n-p+1$, we have that $\Delta f_k= -(n-p)f_k$.
\item[$(iv)$] For each $k,l=1,\ldots,n-p+1$ with $k\neq l$, we have that
\begin{gather*}
f_k f_l +\langle\nabla f_k,\nabla f_l\rangle=0,
\\
f_k^2+|\nabla f_k|^2=1
\end{gather*}
$\mf$-a.e.\ in $M$.
\end{itemize}
\end{Lem}
\begin{proof}
Clearly, $f_k$ is a Lipschitz function.
If we get $(iii)$, we have $\Delta f_k=-(n-p)f_k\in W^{1,2}(M)$, and so we have $f_k\in \TestF(M)$.

We first show that $\{f_{k,i}\}$ strongly converges to $f_k$ as $i\to \infty$ at each point $z\in M$.
Note that by the gradient estimate for eigenfunctions \cite[Theorem~7.3]{Pe3}, there exists a constant $C>0$ such that
$\|f_{k,i}\|_{L^\infty}+\|\nabla f_{k,i}\|_{L^\infty}\leq C$ holds for all $i\in \Z_{>0}$ and $k=1,\ldots, n-p+1$.
Take arbitrary $z=(u,x)\in S^{n-p}(1)\times X$ and $z_i\in M_i$ ($i\in \Z_{>0}$) with $z_i \stackrel{{\rm GH}}{\to} z$.
Since we have $\Psi_i(z_i)\to u$ in~$S^{n-p}(1)$, we have that
\[
\lim_{i\to \infty}\bigg|
\frac{f_{k,i}(z_i)}{|(f_{1,i},\ldots, f_{n-p+1,i})|(z_i)}-f_k(z)
\bigg|=0,
\]
and so $\lim_{i\to \infty}|f_{k,i}(z_i)-f_k(z)|=0$ by Proposition~\ref{prev}$(iii)$.
Since $f_{k,i}$ and $f_k$ are Lipschitz functions whose Lipschitz constants are bounded independently of $i$, we have that
\begin{gather*}
\frac{1}{\nHa^n(B_r (z_i))}\int_{B_r(z_i)}\bigg|f_{k,i}-\frac{1}{\mf(B_r(z))}\int_{B_r(z)}f_k\,{\rm d}\mf\bigg|\,{\rm d}\nHa^n\leq\frac{1}{\nHa^n(B_r (z_i))}
\\ \qquad\qquad
\times\int_{B_r(z_i)}\!\!\bigg(\!|f_{k,i}\!-\!f_{k,i}(z_i)|\!+\!|f_{k,i}(z_i)\!-\!f_k(z)|
+\frac{1}{\mf(B_r(z))}\!\int_{B_r(z)}\!\!|f_k-f_k(z)|\,{\rm d}\mf\!\bigg){\rm d}\nHa^n
\\ \qquad
\leq C r + |f_{k,i}(z_i)-f_k(z)|.
\end{gather*}
Thus,
\begin{gather*}
\lim_{r\to 0}\limsup_{i\to \infty}\frac{1}{\nHa^n(B_r (z_i))}\int_{B_r(z_i)}\bigg|f_{k,i}-\frac{1}{\mf(B_r(z))}\!\int_{B_r(z)}f_k\,{\rm d}\mf\bigg|\,{\rm d}\nHa^n=0.
\end{gather*}
Similarly, we have
\[
\lim_{r\to 0}\limsup_{i\to \infty}\frac{1}{\mf^n(B_r (z))}\int_{B_r(z)}\bigg|f_{k}-\frac{1}{\nHa^n(B_r(z_i))}\int_{B_r(z_i)}f_{k,i}\,{\rm d}\nHa^n\bigg|\,{\rm d}\mf=0.
\]
Therefore, we get $(i)$ by Proposition~\ref{eqconv}.
We get $(ii)$ by Theorem~\ref{convf}$(i)$.
By Theorem~\ref{convf}$(ii)$, we have
\[
\|\Delta f_k+(n-p)f_k\|_{L^2}\leq\liminf_{i\to \infty}\|\Delta f_{k,i}+(n-p)f_{k,i}\|_{L^2}=0,
\]
and so we get $(iii)$.
For each $k=1,\ldots,n-p+1$, we have that
\[
\|f_k^2 +|\nabla f_k|^2-1\|_{L^1}=\lim_{i\to \infty} \|f_{k,i}^2 +|\nabla f_{k,i}|^2-1\|_{L^1},
\]
by \cite[Propositions~3.11 and~3.45]{Ho1} and the original definition of the $L^2$ strong convergence \cite[Definitions~3.25 and~3.58]{Ho1}.
See also \cite[Proposition~3.3 and~Theorem~5.7]{AH}.
By Proposition~\ref{prev}$(i)$, we have $\lim\limits_{i\to \infty} \|f_{k,i}^2 +|\nabla f_{k,i}|^2-1\|_{L^1}=0$, and so we get $f_k^2+|\nabla f_k|^2=1$ $\mf$-a.e.\ in~$M$.
Similarly, applying Proposition~\ref{prev}$(i)$ to $(f_{k,i}\pm f_{l,i})/{\sqrt{2}}$, we get $f_k f_l +\langle\nabla f_k,\nabla f_l\rangle=0$ $\mf$-a.e.\ in $M$ for each $k,l=1,\ldots,n-p+1$ with $k\neq l$.
These imply $(iv)$.
\end{proof}

Let us apply Theorem~\ref{RLF} to vector fields generating rotations in $S^{n-p}(1)$.
\begin{Lem}
Take arbitrary $u=(u_1,\ldots,u_{n-p+1}),v=(v_1,\ldots,v_{n-p+1})\in S^{n-p}(1)$ with $u\cdot v=0$ and $T>0$.
Then, the vector field
\[
V_{u v}:=\sum_{i,j=1}^{n-p+1}u_i v_j(f_i\nabla f_j-f_j\nabla f_i)=f_u\nabla f_v-f_v\nabla f_u
\]
is an element of $\D(\Div)$ with $D^{\sym} V_{u v}\in L^2(M)$ and $\Div V_{u v}=0$.
Moreover, the regular Lag\-ran\-gian flow $\Fl^{V_{u v}}\colon M\times [0,T]\to M$ for $V_{ u v}$ exists and satisfies, for $\mf$-a.e.\ $z\in M$,
\begin{gather*}
f_u(\Fl^{V_{u v}}_t(z))=f_u(z) \cos t- f_v(z)\sin t,
\\
f_v(\Fl^{V_{u v}}_t(z))=f_u(z) \sin t+ f_v(z)\cos t,
\\
f_w(\Fl^{V_{u v}}_t(z))=f_w(z)\qquad (w\in S^{n-p}(1)\text{ with }u\cdot w=v\cdot w=0),
\\
p_2(\Fl^{V_{u v}}_t(z))=p_2(z)
\end{gather*}
for any $t\in[0,T]$.
Moreover,
$\Fl^{V_{u v}}$ preserves the measure $\mf$, i.e.,
\[
\big(\Fl^{V_{u v}}_t\big)_\ast \mf=\mf
\]
for any $t\in[0,T]$.
\end{Lem}

\begin{proof}
Since we have
\[
\Div(f_k\nabla f_l)=-(n-p)f_k f_l+\langle\nabla f_k,\nabla f_l\rangle
\]
for each $k$, $l$ by (2.3.13) in~\cite{Gig}, we get $\Div(V_{u v})=0$.
We next check $D^{\sym} V_{u v}\in L^2(X)$.
It~is enough to show~(\ref{symd}) when $f,g\in\TestF(M)$ because
we have $\|\nabla f -\nabla P_t f\|_{L^4}\to 0$ and $\|\Delta f-\Delta P_t f\|_{L^4}\to 0$ as $t\to 0$ for each $f\in \D(\Delta)$ with $|\nabla f|\in L^4(M)$ and $\Delta f\in L^4(M)$.
Note that
\[
L^4(T M):=\{V\in L^2(T M) \colon |V|\in L^4(M)\}
\]
is a uniformly convex Banach space with the norm $\|V\|_{L^4}:=\||V|\|_{L^4}$, and that $\|\nabla P_t f\|_{L^4}\leq \|\nabla f\|_{L^4}$ for each $f\in W^{1,2}(X)$ with $|\nabla f|\in L^4(X)$ by the $\BL(n-p-1,n)$ condition for $(M,\mf)$.
Combining these and $\nabla P_t f\to \nabla f$ in $L^2(T M)$, we get $\nabla P_t f\to \nabla f$ in $L^4(T M)$.

Take $f,g\in\TestF(M)$.
A simple calculation implies
\begin{gather*}
\int_M \langle V_{u v},\nabla f\rangle \Delta g+ \langle V_{u v},\nabla g\rangle\Delta f -\langle\nabla f,\nabla g\rangle\Div V_{u v}\,{\rm d}\mf
\\ \qquad
{}=\int_M \langle\nabla f_v,\nabla f\rangle \Div(f_u\nabla g)+ \langle \nabla f_v,\nabla g\rangle \Div(f_u\nabla f) +f_u\langle \nabla f_v,\nabla\langle\nabla f,\nabla g\rangle\rangle\, {\rm d}\mf
\\ \phantom{ qquad = }
-\int_M \langle\nabla f_u,\nabla f\rangle \Div(f_v\nabla g)+ \langle \nabla f_u,\nabla g\rangle \Div(f_v\nabla f) +f_v\langle \nabla f_u,\nabla\langle\nabla f,\nabla g\rangle\rangle\, {\rm d}\mf.
\end{gather*}
Note that we have $\langle\nabla f,\nabla g\rangle\in W^{1,2}(M)$ by \cite[Proposition~3.1.3]{Gig}.
By \cite[Definition~3.3.1 and~Theorem~3.3.8]{Gig}, we get $D^{\sym} V_{u v}\in L^2(M)$.
Thus, there exists a regular Lagrangian flow $\Fl^{V_{u v}} \colon M\times [0,T]\to M$ for $V_{ u v}$ by Theorem~\ref{RLF}.

For $\mf$-a.e.\ $z\in M$, we have
\begin{align*}
\frac{\rm d}{{\rm d} t}f_w \left(\Fl^{V_{u v}}(t,z)\right)
=&\left(f_u \langle \nabla f_v,\nabla f_w\rangle-f_v\langle\nabla f_u,\nabla f_w\rangle\right) \left(\Fl^{V_{u v}}(t,z)\right)\\
=&
\begin{cases}
-f_v \big(\Fl^{V_{u v}}(t,z)\big)& (w=u),
\\
f_u \big(\Fl^{V_{u v}}(t,z)\big)& (w=v),
\\
0&  (w\in S^{n-p}(1)\text{ with }u\cdot w=v\cdot w=0)
\end{cases}
\end{align*}
for a.e.\ $t\in(0,T)$ by Lemma~\ref{hfun}.
This implies that for $\mf$-a.e.\ $z\in M$,
{\samepage\[
f_w\left(\Fl^{V_{u v}}(t,z)\right)=
\begin{cases}
f_u(z) \cos t- f_v(z)\sin t & (w=u),
\\
f_u(z) \sin t+ f_v(z)\cos t& (w=v),
\\
f_w(z)& (w\in S^{n-p}(1)\text{ with }u\cdot w=v\cdot w=0)
\end{cases}
\]
for any $t\in [0,T]$.

}

A simple calculation implies that
\[
\langle \nabla f_k,\nabla (g\circ p_2)\rangle=0
\]
$\mf$-a.e.\ in $M$ for each $g\in \LIP(X)$ and $k=1,\ldots, n-p+1$ similarly to Lemma~\ref{prsb} $(iv)$ below.
Therefore, for each $g\in \LIP(X)$ and $\mf$-a.e.\ $z\in M$, we have
\[
\frac{\rm d}{{\rm d} t}(g\circ p_2) \left(\Fl^{V_{u v}}(t,z)\right)=0
\]
for a.e.\ $t\in [0,T]$, and so
\[
(g\circ p_2) \left(\Fl^{V_{u v}}(t,z)\right)=g\circ p_2(z)
\]
for any $t\in[0,T]$.
Let $\{x_j\}_{j\in \Z_{>0}}$ be a countable dense subset of $X$.
Then, by considering $g_j:=d(x_j,\cdot)$, we get that for $\mf$-a.e.\ $z\in M$,
\[
{\rm d}\big(x_j,p_2(\Fl^{V_{u v}}(t,z))\big)={\rm d}(x_j,p_2(z))
\]
holds
for any $j\in\Z_{>0}$ and $t\in [0,T]$.
This implies for $\mf$-a.e.\ $z\in M$,
$p_2\left(\Fl^{V_{u v}}(t,z)\right)=p_2(z)$ for~any $t\in [0,T]$.

We have that $\Fl^{V_{u v}}_t\circ \Fl_t^{-V_{u v}}=\Id_M$ $\mf$-a.e.\ in $M$ for all $t\in[0,T]$ (note that $-V_{u v}=V_{v u}$), and~so
\[
\mf=\big(\Fl^{V_{u v}}_t\big)_\ast \big(\Fl_t^{-V_{u v}}\big)_\ast \mf
\leq \big(\Fl_t^{V_{u v}}\big)_\ast\mf\leq \mf.
\]
This implies the final assertion.
\end{proof}

\begin{Cor}\label{presm}
For any $T\in \mathrm{SO}(n-p+1)$, the transformation
\[
T\colon\ S^{n-p}(1)\times X\to S^{n-p}(1)\times X,\, (u,x)\mapsto (T u, x)
\]
preserves the measure $\mf$.
\end{Cor}
\begin{proof}
Modifying on $\mf$-negligible subset, we have that $\Fl^{V_{u v}}_t\in \mathrm{SO}(n-p+1)$ for each $u,v\in S^{n-p}(1)$ with $u\cdot v=0$ and $t\in [0,2\pi]$.
Conversely, any $T\in \mathrm{SO}(n-p+1)$ can be expressed as a~composition of several transformations of the form $\Fl^{V_{u v}}_t$.
Thus, we get the corollary.
\end{proof}

The following proposition is the goal of this subsection.
\begin{Prop}\label{spms}
Define a Borel measure $\mf_X$ on $X$ by
\[
\mf_X:=(p_2)_\ast \mf.
\]
Then, we have $\mf=\nHa^{n-p}\times \mf_X$.
\end{Prop}
\begin{proof}
We first fix a Borel subset $B\subset X$. Define a Borel measure $\mu_B$ on $S^{n-p}(1)$ by
\[
\mu_B:=(p_1)_\ast \big(\mf|_{p_2^{-1}(B)}\big),
\]
i.e., we define
\[
\mu_B(A):=\mf(A\times B)
\]
for any subset $A\subset S^{n-p}(1)$.
\begin{Clm}\label{presmu}
Each $T\in \mathrm{SO}(n-p+1)$ preserves the measure $\mu_B$.
\end{Clm}
\begin{proof}
We immediately have the claim by Corollary~\ref{presm}.
\end{proof}
\begin{Clm}\label{ac}
$\mu_B\ll \Ha^{n-p}$.
\end{Clm}
\begin{proof}
By the volume estimate relative to $\Ha^{n-p}$ on $S^{n-p}(1)$, there exists a constant $C>0$ such that
\[
\max \left\{k\in\Z_{>0}\colon
\begin{array}{@{}l@{}}\text{there exist $x_1,\ldots,x_k\in S^{n-p}(1)$ such that}\\
\text{$B_r(x_i)\cap B_r(x_j)=\varnothing$ holds for each $i\neq j$}
\end{array}
\right\}\geq r^{-(n-p)}/C
\]
holds for all $r>0$.

Take $r>0$.
We can choose $k\in\Z_{>0}$ with $k\geq r^{-(n-p)}/C$ and $x_1,\ldots, x_k\in S^{n-p}(1)$ such that $B_r(x_i)\cap B_r(x_j)=\varnothing$ holds for each $i\neq j$.
By Claim~\ref{presmu}, we have that
\[
\mu_B(B_r(x_i))=\mu_B (B_r(x_j))=\mu_B(B_r(x))
\]
for all $i,j=1,\ldots,k$ and $x\in S^{n-p}(1)$.
Therefore, we get that $\mu_B(B_r(x))\leq Cr^{n-p}\mu_B(S^{n-p}(1))$ for all $x\in S^{n-p}(1)$.

Take arbitrary subset $A\subset S^{n-p}(1)$ with $\Ha^{n-p}(A)=0$ and $\epsilon>0$.
Then, by the definition of the Hausdorff measure, there exists a sequence of subsets $\{S_j\}_{j\in\Z_{>0}}$ of $S^{n-p}(1)$ such that
$A\subset \bigcup_{j=1}^\infty S_j$ and
\[
\sum_{j=1}^\infty (\diam S_j)^{n-p}<\epsilon.
\]
Choose $x_j\in S_j$ for each $j$.
Then, we have $S_j\subset B_{\diam S_j}(x_j)$, and so
\[
\mu_B(A)\leq \sum_{j=1}^\infty\mu_B(B_{\diam S_j}(x_j))
\leq C\sum_{j=1}^\infty (\diam S_j)^{n-p}\mu_B(S^{n-p}(1))\leq C\epsilon \mu_B\big(S^{n-p}(1)\big).
\]
Letting $\epsilon \to 0$, we obtain $\mu_B(A)=0$ and get the claim.
\end{proof}

By Claim~\ref{ac} and the Radon--Nikodym theorem, we have the representation $\mu_B=\rho \Ha^{n-p}$, where $\rho\colon S^{n-p}(1)\to [0,\infty]$ is some Borel function.
By Claim~\ref{presmu}, we have that for each $T\in \mathrm{SO}(n-p+1)$
\[
\rho\circ T=\rho
\]
$\Ha^{n-p}$-a.e.\ in $S^{n-p}(1)$.
This implies that $\rho$ is constant $\Ha^{n-p}$-a.e.\ in $S^{n-p}(1)$.
We have that
\[
\rho\Ha^{n-p}\big(S^{n-p}(1)\big)=\mu_B\big(S^{n-p}(1)\big)=\mf\big(S^{n-p}(1)\times B\big)=\mf_X(B),
\]
and so $\mu_B=\mf_X(B)\nHa^{n-p}$.

For each Borel sets $A\subset S^{n-p}(1)$ and $B\subset X$, we get that
\[
\mf(A\times B)=\mu_B(A)=\nHa^{n-p}(A)\mf_X(B).
\tag*{\qed}
\]
\renewcommand{\qed}{}
This implies the proposition.
\end{proof}

\subsection[Product metric measure spaces and the RCD* condition]{Product metric measure spaces and the RCD$\boldsymbol{^*}$ condition}
In the previous subsection we showed that there exists a Borel measure $\mf_X$ on X such that $\mf=\nHa^{n-p}\times \mf_X$ holds under Assumption~\ref{convX}.
In this subsection we show that $(X,\mf_X)$ satisfies the $\mathrm{RCD}(n-p-1,p)$ condition.

More generally, we consider the following assumption.
\begin{Asu}\label{Aprod}
Let $K, N\in \R$ with $N\geq 1$ and $(X_i,d_i,\mf_i)$ ($i=1,2$) be compact metric measure spaces.
Put $(M,d,\mf):=(X_1\times X_2,d_1\times d_2,\mf_1\times \mf_2)$.
Moreover, we assume the following:
\begin{itemize}\itemsep=0pt
\item $\mf_i(X_i)=1$ ($i=1,2$),
\item $(M,d,\mf)$ satisfies the $\RCD(K,N)$ condition.
\end{itemize}
For each $i=1,2$, let $p_i\colon M\to X_i$ denote the projection.
\end{Asu}

The goal of this subsection is to prove the following proposition:

\begin{Prop}\label{KN-n}
In addition to Assumption~$\ref{Aprod}$, we assume that $(X_1,d_1,\mf_1)$ is an $n$-dimen\-sio\-nal closed Riemannian manifold with the Riemannian distance and $\mf_1=\nHa^n$.
Then, \linebreak $(X_2,d_2,\mf_2)$ satisfies the $\RCD(K,N-n)$ condition if $N-n\geq 1$.
\end{Prop}

Note that if $(Y_i,d'_i,\mf'_i)$ ($i=1,2$) are $\RCD(K,N_i)$ spaces, then the product space
$(Y_1\times Y_2,\allowbreak d'_1\times d'_2,\mf'_1\times \mf_2')$ satisfies the $\RCD(K,N_1+N_2)$ condition by \cite[Theorem~3.23]{EKS}.

We first show the following easy lemma.
\begin{Lem}\label{prsb}
Under Assumption~$\ref{Aprod}$, we have the following properties:
\begin{itemize}\itemsep=0pt
\item[$(i)$] For any $f\in \LIP(M)$ and $x=(x_1,x_2)\in M$, we have that
\[
\Lip_M(f)(x)\geq \left(\Lip_{X_1} (f(\cdot,x_2))\right)(x_1)
\]
and that
\[
\Lip_M(f)(x)\geq \left(\Lip_{X_2} (f(x_1,\cdot))\right)(x_2).
\]
\item[$(ii)$] For any $f\in \LIP(X_i)$ $(i=1,2)$, we have that
\[
(\Lip_{X_i} f)\circ p_i=\Lip_M (f\circ p_i).
\]
\item[$(iii)$] For each $i=1,2$, the map $p_i^\ast\colon L^2(X_i,\mf_i)\to L^2(M,\mf),\,f\mapsto f\circ p_i$ induces an isometric immersion $p_i^\ast \colon W^{1,2}(X_i)\to W^{1,2}(M)$,
and we have that
\[
|\nabla (f\circ p_i)|=|\nabla f|\circ p_i
\]
$\mf$-a.e.\ in $M$ for any $f\in W^{1,2}(X_i)$ $(i=1,2)$.
\item[$(iv)$] For any $f\in W^{1,2}(X_1)$ and $h\in \LIP(M)$, we have that
\[
\langle\nabla (f\circ p_1),\nabla h\rangle(x)=\langle\nabla f,\nabla (h(\cdot,x_2))\rangle(x_1)
\]
for $\mf$-a.e.\ $x=(x_1,x_2)\in M$.
The similar result holds for the element of $W^{1,2}(X_2)$.
\item[$(v)$] For any $f_i\in W^{1,2}(X_i)$ $(i=1,2)$, we have that
\[
\langle\nabla (f_1\circ p_1),\nabla (f_2\circ p_2)\rangle=0
\]
$\mf$-a.e.\ in $M$.
\end{itemize}
\end{Lem}
\begin{proof}
We get $(i)$ and $(ii)$ straightforward by the definition.

We show $(iii)$ for $i=1$.
Take arbitrary $f\in W^{1,2}(X_1)$.
For any sequence $\{f_n\}\subset \LIP(X_1)$ with $\lim\limits_{n\to \infty}\|f_n-f\|_{L^2}=0$, we have
$\lim\limits_{n\to \infty}\|f_n\circ p_1-f\circ p_1\|_{L^2}=0$, and so
\[
\Ch_M(f\circ p_1)\leq\frac{1}{2} \liminf_{n\to \infty}\int_M (\Lip_M(f_n\circ p_1))^2\,{\rm d}\mf=\frac{1}{2}\liminf_{n\to \infty}\int_{X_1}(\Lip_{X_1} f_n)^2\,{\rm d}\mf_1
\]
by $(ii)$. This implies $\Ch_M(f\circ p_1)\leq \Ch_{X_1}(f)$ and $f\circ p_1\in W^{1,2}(M)$.

We next show $\Ch_M(f\circ p_1)\geq \Ch_{X_1}(f)$.
Take any sequence $\{f_n\}\subset\LIP(M)$ with $\epsilon_n:=\|f_n-f\circ p_1\|_{L^2}^2\to 0$ as $n\to\infty$.
We can assume $\epsilon_n<1$ for each $n$.
We have that
\begin{gather*}
\epsilon_n=\int_{X_2} \int_{X_1} (f_n(x_1,x_2)-f(x_1))^2\,{\rm d}\mf_1(x_1)\,{\rm d}\mf_2(x_2),
\\
\frac{1}{2} \int_M (\Lip_M f_n)^2\,{\rm d}\mf\geq \frac{1}{2}\int_{X_2} \int_{X_1} (\Lip_{X_1} f_n(\cdot,x_2))^2(x_1)\,{\rm d}\mf_1(x_1)\,{\rm d}\mf_2(x_2)
\end{gather*}
by $(i)$,
and so
\begin{gather*}
\mf_2\bigg(\bigg\{x_2\in X_2\colon \int_{X_1}(f_n(\cdot,x_2)-f)^2\,{\rm d}\mf_1> 2\epsilon_n^{1/2}\bigg\}\bigg)\leq \frac{1}{2}\epsilon_n^{1/2},
\\
\mf_2\bigg(\bigg\{x_2\in X_2\colon \int_{X_1}(\Lip_{X_1}(f_n(\cdot,x_2)))^2\,{\rm d}\mf_1> (1+\epsilon_n^{1/2})\int_M(\Lip_M f_n)^2\,{\rm d}\mf\bigg\}\bigg)\leq \frac{1}{1+\epsilon^{1/2}}.
\end{gather*}
Since we have
\[
\frac{1}{2}\epsilon_n^{1/2}+\frac{1}{1+\epsilon_n^{1/2}}<1,
\]
we can take a sequence $\{x_2(n)\}\subset X_2$ such that
\begin{gather*}
\int_{X_1}(f_n(\cdot,x_2(n))-f)^2\,{\rm d}\mf_1\leq 2\epsilon_n^{1/2},
\\
\int_{X_1}(\Lip_{X_1}(f_n(\cdot,x_2(n))))^2\,{\rm d}\mf_1
\leq  (1+\epsilon_n^{1/2})\int_M(\Lip_M f_n)^2\,{\rm d}\mf
\end{gather*}
for each $n$.
Put $g_n:=f_n(\cdot, x_2(n))\in\LIP(X_1)$.
Then, we have $\|g_n-f\|_{L^2}\to 0$ as $n\to\infty$ and
\[
\Ch_{X_1}(f)\leq \frac{1}{2}\liminf_{n\to \infty}\int_{X_1}(\Lip_{X_1}g_n)^2\,{\rm d}\mf_1\leq \frac{1}{2}\liminf_{n\to \infty}\int_M(\Lip_M f_n)^2\,{\rm d}\mf.
\]
Thus, we get $\Ch_M(f\circ p_1)\!\geq\! \Ch_{X_1}(f)$, and so $\Ch_M(f\circ p_1)\!=\! \Ch_{X_1}(f)$.
Therefore, $p_1^\ast\colon W^{1,2}(X_1)\!\to W^{1,2}(M)$ is isometric

Let us show that $|\nabla (f\circ p_1)|=|\nabla f|\circ p_1$ $\mf$-a.e.\ in $M$.
Take $f_n\in \LIP(X_1)$ such that $f_n\to f$ and $\Lip_{X_1} f_n\to |\nabla f|$ in $L^2$.
Then, we have $f_n\circ p_1\to f\circ p_1$, $\Lip_M (f_n\circ p_1)=(\Lip_{X_1} f_n)\circ p_1 \to |\nabla f|\circ p_1$ in $L^2$ and
\[
\Ch_{M}(f\circ p_1)=\Ch_{X_1}(f)=\frac{1}{2}\int_{M}(|\nabla f|\circ p_1)^2\,{\rm d}\mf.
\]
This implies $|\nabla (f\circ p_1)|=|\nabla f|\circ p_1$ $\mf$-a.e.\ in $M$.

Let us prove $(iv)$.
We first consider the case $f\in \Lip(X_1)$.
Then, we have
\begin{gather*}
\begin{split}
\langle\nabla(f\circ p_1),\nabla h\rangle(x_1,x_2)
&=\lim_{\epsilon\to 0}\frac{1}{2\epsilon}\left(\Lip_M(f\circ p_1+\epsilon h)^2-\Lip_M(f\circ p_1)^2\right)(x_1,x_2)\\
&\geq \lim_{\epsilon\to 0}\frac{1}{2\epsilon}\left(\Lip_{X_1}(f+\epsilon h(\cdot,x_2))^2-\Lip_{X_1}(f)^2\right)(x_1)\\
&=\langle\nabla f,\nabla (h(\cdot,x_2))\rangle(x_1)
\end{split}
\end{gather*}
for $\mf$-a.e.\ $(x_1,x_2)\in M$.
By considering $-h$ instead of $h$, we also get \[\langle\nabla(f\circ p_1),\nabla h\rangle(x_1,x_2)\leq \langle\nabla f,\nabla (h(\cdot,x_2))\rangle(x_1),\] and so \[\langle\nabla(f\circ p_1),\nabla h\rangle(x_1,x_2)= \langle\nabla f,\nabla (h(\cdot,x_2))\rangle(x_1)\]
for $\mf$-a.e.\ $(x_1,x_2)\in M$.
For general $f\in W^{1,2}(X_1)$, approximating $f$ by Lipschitz functions, we~get $(iv)$.

Finally we show $(v)$.
We have $(v)$ for each $f_i\in\LIP(X_i)$ ($i=1,2$) by $(iv)$.
For general $f_i\in W^{1,2}(X_i)$, approximating $f_i$ by Lipschitz functions, we get $(v)$.
\end{proof}

We immediately get the following corollary by Lemma~\ref{prsb}$(iii)$.
\begin{Cor}\label{ihsl}
Under Assumption~$\ref{Aprod}$, we have that the metric measure space $(X_i,d_i,m_i)$ is infinitesimally Hilbertian and satisfies the Sobolev-to-Lipschitz property for each $i=1,2$.
\end{Cor}

For any $f_i\in L^2(X_i)$, we shall denote $f_i\circ p_i\in L^2(M)$ by $f_i$ briefly if there is no confusion.
\begin{Lem}\label{sbprod}
Under Assumption~$\ref{Aprod}$, we have $f_1 f_2\in W^{1,2}(M)$ and
\[
\nabla (f_1 f_2)=f_1 \nabla f_2+f_2 \nabla f_1\in L^2(T M)
\]
for any $f_i\in W^{1,2}(X_i)$ $(i=1,2)$.
\end{Lem}
\begin{proof}
Take sequences $\{f_{i,n}\}_{n\in \Z_{>0}}\subset \LIP(X_i)$ ($i=1,2$) such that
$f_{i,n}\to f_i$ and $\Lip_{X_i}(f_{i,n})\to |\nabla f_{i,n}|$ in $L^2(X_i)$ as $n\to \infty$.
We have $f_{1,n} f_{2,n}\in W^{1,2}(M)$ and $\nabla (f_{1,n} f_{2,n})=f_{1,n}\nabla f_{2,n} +f_{2,n}\nabla f_{1,n}$ by \cite[Theorem~2.2.6]{Gig}.
Then, $f_{1,n} f_{2,n}\to f_1 f_2$ in $L^2(M)$ and
\[
\nabla (f_{1,n} f_{2,n})=f_{1,n}\nabla f_{2,n} +f_{2,n}\nabla f_{1,n}\to f_1\nabla f_2+f_2\nabla f_1\]
in $L^2(T M)$.
Thus, we get that $f_{1,n} f_{2,n}\to f_1 f_2$ in $W^{1,2}(M)$ and
$\nabla (f_1 f_2)=f_1 \nabla f_2+f_2 \nabla f_1\in L^2(T M)$.
\end{proof}

Let us consider the Laplacian on $M$.
\begin{Lem}\label{Laprod}
Under Assumption~$\ref{Aprod}$, we have the following properties:
\begin{itemize}\itemsep=0pt
\item[$(i)$] For each $i=1,2$, the map $p_i^\ast\colon L^2(X_i)\to L^2(M)$ induces a map $p_i^\ast \colon \D(\Delta_{X_i})\to \D(\Delta_M)$, and we have that
\[
(\Delta_{X_i}f)\circ p_i=\Delta_M(f\circ p_i)
\]
for any $f\in \D(\Delta_{X_i})$ $(i=1,2)$.
Thus, we use the same notation $\Delta$ for $\Delta_M$ and $\Delta_{X_i}$ $(i=1,2)$.
For any $f\in \D(\Delta_{X_i})$, we shall denote $(\Delta_{X_i} f)\circ p_i$ by $\Delta f$ briefly if there is no confusion.
\item[$(ii)$] For any $f_i\in\D(\Delta_{X_i})$ $(i=1,2)$, we have that $f_1 f_2\in\D(\Delta_M)$ and that
\[
\Delta(f_1 f_2)=f_1\Delta f_2+(\Delta f_1) f_2.
\]
\end{itemize}
\end{Lem}
\begin{proof}
We show $(i)$ for $i=1$.
Take arbitrary $f\in \D(\Delta_{X_1})$.
Then, for any $\phi\in \LIP(M)$, we have
\begin{align*}
\int_M\langle\nabla (f\circ p_1),\nabla\phi\rangle
=&\int_{X_2}\int_{X_1}\langle\nabla f,\nabla (\phi(\cdot,x_2))\rangle(x_1)\,{\rm d}\mf_1(x_1)\,{\rm d}\mf_2(x_2)\\
=&-\int_{X_2}\int_{X_1} \Delta_{X_1} f(x_1)\phi(x_1,x_2)\,{\rm d}\mf_1(x_1)\,{\rm d}\mf_2(x_2)\\
=&-\int_M (\Delta_{X_1} f)\circ p_1 \cdot \phi \,{\rm d}\mf .
\end{align*}
Since $\LIP(M)\subset W^{1,2}(M)$ is dense with respect to the norm $\|\cdot\|_{W^{1,2}}$, we get $(i)$.

We next show $(ii)$.
Take arbitrary $f_i\in\D(\Delta_{X_i})$ ($i=1,2$).
Then, for any $\phi\in\LIP(M)$, we have
\begin{align*}
\int_M\langle\nabla (f_1 f_2),\nabla\phi\rangle=&\int_{X_1}\int_{X_2}f_1(x_1)\langle\nabla f_2,\nabla (\phi(x_1,\cdot))\rangle(x_2)\,{\rm d}\mf_1(x_1)\,{\rm d}\mf_2(x_2)
\\
&{}+ \int_{X_1}\int_{X_2}f_2(x_2)\langle\nabla f_1,\nabla (\phi(\cdot, x_2))\rangle(x_1)\,{\rm d}\mf_1(x_1)\,{\rm d}\mf_2(x_2)
\\
=&-\int_M (f_1\Delta f_2+(\Delta f_1)f_2) \phi \,{\rm d}\mf.
\end{align*}
Since $\LIP(M)\subset W^{1,2}(M)$ is dense, we get $(ii)$.
\end{proof}

Our goal is to show $(X_2,d_2,\mf_2)$ satisfies the $\RCD(K,N-n)$ condition under the assumption of Proposition~\ref{KN-n}.
However, we can show the following weaker assertion under Assumption~\ref{Aprod}.
\begin{Cor}
Under Assumption~$\ref{Aprod}$, we have that the metric measure space $(X_i,d_i,\mf_i)$ satisfies the $\RCD(K,N)$ condition for each $i=1,2$.
\end{Cor}
\begin{proof}
We only need to show that $(X_i,d_i,\mf_i)$ satisfies the $\BE(K,N)$ condition by Corollary~\ref{ihsl}.
For any $u_i\in\D(\Delta_{X_i})$ with $\Delta u_i\in W^{1,2}(X_i)$ and $\phi_i\in\D(\Delta_{X_i})\cap L^\infty(X_i)$ with $\phi_i\geq 0$ and $\Delta_{X_i}\phi_i\in L^\infty(X_i)$, applying the $\BE(K,N)$ condition for $(M,d,\mf)$ to $u_i\circ p_i, \phi_i\circ p_i\in\D(\Delta_M)$, we get the $\BE(K,N)$ condition for $(X_i,d_i,\mf_i)$.
\end{proof}

The following proposition is crucial to show Proposition~\ref{KN-n}.
We show the $\BE(K,N-n)$ condition with an error term.
\begin{Prop}\label{preBE}
In addition to Assumption~$\ref{Aprod}$, we assume that $n$ is an integer with $N-n\geq 1$ and that $(X_1,d_1,\mf_1)$ is an $n$-dimensional closed Riemannian manifold with the Riemannian distance and $\mf_1=\nHa^n$.
Then, for all $u\in\D(\Delta_{X_2})$ with $\Delta u\in W^{1,2}(X_2)$ and all $\phi\in\D(\Delta_{X_2})\cap L^\infty(X_2)$ with $\phi\geq0$ and $\Delta \phi\in L^\infty(X_2)$, we have
\begin{gather*}
\frac{1}{2}\int_{X_2}\Delta \phi|\nabla u|^2 \,{\rm d}\mf_2
\\ \qquad
{}\geq \int_{X_2} \phi\bigg(\langle\nabla \Delta u,\nabla u\rangle+K|\nabla u|^2+\frac{(\Delta u)^2}{N-n}-\frac{2n}{N(N-n)}(\Delta u-2(N-n)u)^2\bigg) {\rm d}\mf_2.
\end{gather*}
\end{Prop}
\begin{proof}
Take $\psi\in C^\infty(\R)$ such that
\[
\psi(t)=
\begin{cases}
1,& |t|\leq \dfrac{1}{2},
\\
0,& |t|\geq 1
\end{cases}
\]
and $\psi\geq 0$.
Fix $p\in X_1$.
Take sufficiently small $\epsilon>0$ so that we can take $\psi_\epsilon$ and $f_\epsilon$ below as~smooth functions.
Define $\psi_\epsilon \in C^\infty(X_1)$ by
\[
\psi_\epsilon (x_1):=\psi\left(\frac{d(p,x_1)}{\epsilon}\right)
\]
for each $x_1\in X_1$, and take $f_\epsilon\in C^\infty(X_1)$ such that
\[
f_\epsilon (x_1)=
\begin{cases}
1+d_1(p,x_1)^2,& d(p,x_1) \leq \epsilon,
\\
0,& d(p,x_1)\geq 2\epsilon.
\end{cases}
\]
Then, there exists a constant $C>0$ such that
\begin{gather}\label{hessf}
\begin{array}{l}
\left|\Hess f_\epsilon-\frac{\Delta f_\epsilon}{n}g_{X_1}\right|(x_1)\leq C d(p,x_1),
\\[1ex]
|\Delta f_\epsilon -2n|(x_1)\leq C d(p,x_1),
\\
|\nabla f_\epsilon|(x_1)\leq C d(p,x_1),
\\[1ex]
|f_\epsilon-1|(x_1)\leq C d(p,x_1)
\end{array}
\end{gather}
for all $x_1\in B_\epsilon(p)$.
Note that we can take such a constant independently of $\epsilon$.

\begin{Clm}\label{eper}
There exists a constant $C>0$ such that for all sufficiently small $\epsilon>0$, all $u\in\D(\Delta_{X_2})$ with $\Delta u\in W^{1,2}(X_2)$ and all $\phi\in\D(\Delta_{X_2})\cap L^\infty(X_2)$ with $\phi\geq0$ and $\Delta \phi\in L^\infty(X_2)$, we have
\begin{align*}
\frac{1}{2}\int_{X_2}\Delta \phi|\nabla u|^2 \,{\rm d}\mf_2
\geq &\int_{X_2} \phi\bigg(\langle\nabla \Delta u,\nabla u\rangle+K|\nabla u|^2+\frac{(\Delta u)^2}{N-n}\bigg){\rm d}\mf_2
\\
&{}-C\epsilon^2\int_{X_2} \phi (u^2+|\nabla u|^2+(\Delta u)^2)\,{\rm d}\mf_2
\\
&{}-\frac{2n}{N(N-n)}\int_{X_2}\phi(\Delta u-2(N-n)u)^2{\rm d}\mf_2.
\end{align*}
\end{Clm}

\begin{proof}
We have that $f_\epsilon u\in\D(\Delta_M)$ with $\Delta(f_\epsilon u)=f_\epsilon \Delta u+(\Delta f_\epsilon)u\in W^{1,2}(M)$ and that $\psi_\epsilon \phi\in \D(\Delta_M)\cap L^\infty(M)$ with $\psi_\epsilon\phi\geq 0$ and $\Delta(\psi_\epsilon \phi)=\psi_\epsilon \Delta \phi+(\Delta \psi_\epsilon)\phi\in L^\infty (M)$ by Lemma~\ref{sbprod} and Lemma~\ref{Laprod}.
Thus, we can apply the $\BE(K,N)$ condition to the pair $(f_\epsilon u, \psi_\epsilon \phi)$ and get
\begin{gather}
\frac{1}{2}\int_M \Delta(\psi_\epsilon \phi)|\nabla(f_\epsilon u)|^2\,{\rm d}\mf
\geq\int_M (\psi_\epsilon \phi)\langle\nabla \Delta(f_\epsilon u),\nabla(f_\epsilon u) \rangle\,{\rm d}\mf\nonumber
\\ \hphantom{\frac{1}{2}\int_M \Delta(\psi_\epsilon \phi)|\nabla(f_\epsilon u)|^2\,{\rm d}\mf
\geq}
\label{3a}
+K\int_M \psi_\epsilon \phi |\nabla (f_\epsilon u)|^2\,{\rm d}\mf+
\frac{1}{N}\int_M \psi_\epsilon \phi(\Delta(f_\epsilon u))^2\,{\rm d}\mf.
\end{gather}
We calculate each terms.

We have
\begin{gather}
\frac{1}{2}\int_M \Delta(\psi_\epsilon \phi)|\nabla(f_\epsilon u)|^2\,{\rm d}\mf\nonumber
\\ \qquad
=\int_{X_1}\psi_\epsilon\big(\langle\nabla \Delta f_\epsilon\nabla f_\epsilon\rangle+\Ric(\nabla f_\epsilon,\nabla f_\epsilon)+|\Hess f_\epsilon|^2\big)\,{\rm d}\mf_1
\int_{X_2}\phi u^2\,{\rm d}\mf_2\nonumber
\\ \qquad\quad
{}+\int_{X_1}\!\! \psi_\epsilon |\nabla f_\epsilon|^2\,{\rm d}\mf_1\int_{X_2}\!\!\phi u\Delta u\,{\rm d}\mf_2
+\int_{X_1} \!\!\psi_\epsilon (f_\epsilon \Delta f_\epsilon +2|\nabla f_\epsilon|^2)\,{\rm d}\mf_1\int_{X_2}\!\!\phi|\nabla u|^2\,{\rm d}\mf_2
\nonumber
\\ \qquad\quad{}
+\frac{1}{2}\int_{X_1}\psi_\epsilon f_\epsilon^2 \,{\rm d}\mf_1\int_{X_2}\Delta \phi |\nabla u|^2\,{\rm d}\mf_2.
\label{3b}
\end{gather}
Here, we used the Bochner formula
\[\frac{1}{2}\Delta|\nabla f_\epsilon|^2=\langle\nabla \Delta f_\epsilon\nabla f_\epsilon\rangle+\Ric(\nabla f_\epsilon,\nabla f_\epsilon)+|\Hess f_\epsilon|^2\]
and the equation
\[
\int_{X_2} (\Delta \phi )u^2\,{\rm d}\mf_2=2\int_{X_2} \phi u\Delta u\,{\rm d}\mf_2+2\int_{X_2}\phi|\nabla u|^2\,{\rm d}\mf_2,
\]
which can be justified by approximating $u$ by $P_t u$.

We have
\begin{gather}
\int_M \psi_\epsilon \phi\langle\nabla \Delta(f_\epsilon u),\nabla(f_\epsilon u) \rangle\,{\rm d}\mf
\nonumber
\\ \qquad
=\int_{X_1}\psi_\epsilon\langle\nabla \Delta f_\epsilon,\nabla f_\epsilon\rangle\,{\rm d}\mf_1\int_{X_2}\phi u^2\,{\rm d}\mf_2
+\int_{X_1} \psi_\epsilon |\nabla f_\epsilon|^2\,{\rm d}\mf_1\int_{X_2}\phi u\Delta u\,{\rm d}\mf_2
\nonumber
\\  \qquad\quad
{}+\int_{X_1} \psi_\epsilon f_\epsilon \Delta f_\epsilon\,{\rm d}\mf_1\int_{X_2}\phi|\nabla u|^2\,{\rm d}\mf_2
+\int_{X_1} \psi_\epsilon f_\epsilon^2\,{\rm d}\mf_1\int_{X_2}\phi\langle\nabla \Delta u,\nabla u\rangle\,{\rm d}\mf_2.
\label{3c}
\end{gather}

We have
\begin{gather}
K\int_M \psi_\epsilon \phi|\nabla (f_\epsilon u)|^2\,{\rm d}\mf\nonumber
\\ \qquad
{}=K \int_{X_1}\psi_\epsilon|\nabla f_\epsilon|^2\,{\rm d}\mf_1\int_{X_2}\phi u^2\,{\rm d}\mf_2
+K\int_{X_1} \psi_\epsilon f_\epsilon^2\,{\rm d}\mf_1\int_{X_2}\phi |\nabla u|^2\,{\rm d}\mf_2.
\label{3d}
\end{gather}

We have
\begin{gather}
\frac{1}{N}\int_M \psi_\epsilon \phi(\Delta (f_\epsilon u))^2\,{\rm d}\mf
\nonumber
\\ \qquad
{}=\frac{1}{N}\int_{X_1}\psi_\epsilon(\Delta f_\epsilon)^2\,{\rm d}\mf_1\int_{X_2}\phi u^2\,{\rm d}\mf_2
+\frac{2}{N}\int_{X_1} \psi_\epsilon f_\epsilon \Delta f_\epsilon\,{\rm d}\mf_1\int_{X_2}\phi u \Delta u\,{\rm d}\mf_2
\nonumber
\\ \qquad\quad
{}+\frac{1}{N}\int_{X_1} \psi_\epsilon f_\epsilon^2\,{\rm d}\mf_1\int_{X_2}\phi (\Delta u)^2\,{\rm d}\mf_2.
\label{3e}
\end{gather}

Take $\widetilde{K}>0$ such that $\Ric_{X_1}\leq \widetilde{K} g_{X_1}$.
Then, we get
\begin{gather*}
\frac{1}{2}\int_{X_1}\psi_\epsilon f_\epsilon^2 \,{\rm d}\mf_1\int_{X_2}
\Delta \phi |\nabla u|^2\,{\rm d}\mf_2
\nonumber
\\ \qquad
{}\geq -\int_{X_1}\psi_\epsilon\bigg((\widetilde{K}-K)|\nabla f_\epsilon|^2 +\bigg|\Hess f_\epsilon-\frac{\Delta f_\epsilon}{n}g_{X_1}\bigg|^2\bigg)\,{\rm d}\mf_1\int_{X_2}\phi u^2\,{\rm d}\mf_2
\nonumber
\\  \qquad\quad
{}-2\int_{X_1} \psi_\epsilon |\nabla f_\epsilon|^2\,{\rm d}\mf_1\int_{X_2}\phi|\nabla u|^2\,{\rm d}\mf_2
\nonumber
\\ \qquad\quad
{}-\frac{1}{N n (N-n)}\int_M \psi_\epsilon \phi ((N-n)(\Delta f_\epsilon) u- n f_\epsilon \Delta u)^2\,{\rm d}\mf
\\  \qquad\quad
{}+\int_{X_1} \psi_\epsilon f_\epsilon^2\,{\rm d}\mf_1\int_{X_2} \phi \bigg(\langle\nabla \Delta u,\nabla u\rangle+K|\nabla u|^2+\frac{1}{N-n} (\Delta u)^2\bigg)\,{\rm d}\mf_2
\end{gather*}
by (\ref{3a})--(\ref{3e}).
Since we have
\[
0< \int_{X_1}\psi_\epsilon \,{\rm d}\mf_1\leq \int_{X_1}\psi_\epsilon f_\epsilon^2\,{\rm d}\mf_1\leq \big(1+\epsilon^2\big)^2\int_{X_1}\psi_\epsilon \,{\rm d}\mf_1
\]
and
\begin{gather*}
\big\|\sqrt{\psi_\epsilon \phi}((N-n)(\Delta f_\epsilon) u- n f_\epsilon \Delta u-2n(N-n)u+n\Delta u)\big\|_{L^2}
\\ \qquad
{}\leq C\epsilon \big(\big\|\sqrt{\phi} u\big\|_{L^2}+\big\|\sqrt{\phi} \Delta u\big\|_{L^2}\big)\bigg(\int_{X_1}\psi_\epsilon \,{\rm d}\mf_1\bigg)^{1/2},
\end{gather*}
we get the claim by (\ref{hessf}).
\end{proof}

Letting $\epsilon\to 0$ in Claim~\ref{eper}, we get the proposition.
\end{proof}

Let us complete the proof of Proposition~\ref{KN-n}.
Since we have already showed Corollary~\ref{ihsl}, we only need to check the $\BL(K,N-n)$ condition for $(X_2,d_2,\mf_2)$.
The proof of the following proposition has been inspired by the proof of \cite[Theorem~1.2]{Ket1}.

\begin{Prop}\label{blN-n}
In addition to Assumption~$\ref{Aprod}$, we assume that $n$ is an integer with $N-n\geq 1$ and that $(X_1,d_1,\mf_1)$ is an $n$-dimensional closed Riemannian manifold with the Riemannian distance and $\mf_1=\nHa^n$.
Then, the metric measure space $(X_2,d_2,\mf_2)$ satisfies the $\BL(K,N-n)$ condition.
\end{Prop}
\begin{proof}
Similarly to the proof of the assertion that the $\BE(K,N)$ condition implies the $\BL(K,N)$ condition \cite[Proposition~4.9]{EKS}, we have the following claim:
\begin{Clm}\label{blmod}
For any $u\in\D(\Delta_{X_2})$ and $t>0$, we have
\[
|\nabla P_t u|^2+\frac{1-{\rm e}^{-2Kt}}{K(N-n)}\bigg((\Delta P_t u)^2-\frac{2n}{N}P_t\big((\Delta u-2(N-n)u)^2\big)\bigg)\leq {\rm e}^{-2Kt}P_t\big(|\nabla u|^2\big)
\]
$\mf_2$-a.e.\ in $X_2$.
\end{Clm}
\begin{proof}
Take arbitrary $\phi\in L^\infty(X_2)$ with $\phi\geq 0$.
Define $h\colon [0,t]\to \R$ by
\[
h(s):={\rm e}^{-2Ks} \int_{X_2} P_s\phi|\nabla P_{t-s} u|^2\,{\rm d}\mf_2.
\]
Then, for each $0<s<t$, we have
\begin{gather*}
\frac{\partial}{\partial s} h(s)
=-2K {\rm e}^{-2 K s}\int_{X_2} P_s \phi|\nabla P_{t-s} u|^2\,{\rm d}\mf_2
\\ \hphantom{\frac{\partial}{\partial s} h(s)=}
{}+{\rm e}^{-2K s}\int_{X_2} \Delta P_s \phi |\nabla P_{t-s} u|^2\,{\rm d}\mf_2
-2 {\rm e}^{-2K s}\int_{X_2} P_s\phi\langle\nabla \Delta P_{t-s} u,\nabla P_{t-s}u \rangle\,{\rm d}\mf_2
\\ \hphantom{\frac{\partial}{\partial s} h(s)}
{}\geq 2 \frac{{\rm e}^{-2K s}}{N-n}\bigg(\int_{X_2}\!\!P_s \phi (\Delta P_{t-s} u)^2\,{\rm d}\mf_2
-\frac{2n}{N}\!\int_{X_2}\!\!P_s \phi (\Delta P_{t-s} u-2(N-n)P_{t-s} u)^2{\rm d}\mf_2 \!\bigg)
\\ \hphantom{\frac{\partial}{\partial s} h(s)}
{}\geq 2 \frac{{\rm e}^{-2K s}}{N-n}\bigg(\int_{X_2}\phi (\Delta P_{t} u)^2\,{\rm d}\mf_2
-\frac{2n}{N}\int_{X_2} \phi P_t\big((\Delta u-2(N-n)u)^2\big)\,{\rm d}\mf_2 \bigg).
\end{gather*}
Here, we used
\begin{gather*}
\frac{\partial}{\partial s} P_s \phi=\Delta P_s \phi,
\\
\frac{\partial }{\partial s}P_{t-s} u= -\Delta P_{t-s} u
\end{gather*}
in $W^{1,2}$, $|\nabla P_{t-s} u|^2\leq {\rm e}^{-2K (t-s)}P_{t-s}(|\nabla u|^2)$ (by the $\BL(K,N)$ condition), Proposition~\ref{preBE} and the Jensen inequality
\begin{gather*}
 (\Delta P_{t} u)^2\leq P_{s}\big((\Delta P_{t-s} u)^2\big),
 \\
 (\Delta P_{t-s} u-2(N-n)P_{t-s} u)^2\leq P_{t-s} \big( (\Delta u-2(N-n)u)^2\big)
\end{gather*}
$\mf_2$-a.e.\ in $X_2$. Combining this and
\begin{gather*}
\lim_{s\to 0} h(s)=h(0)=\int_{X_2}\phi |\nabla P_t u|^2\,{\rm d}\mf_2,
\\
\lim_{s\to t} h(s)=h(t)={\rm e}^{-2K t}\int_{X_2}\phi P_t |\nabla u|^2\,{\rm d}\mf_2,
\end{gather*}
we get
\begin{gather*}
\int_{X_2}\phi \left({\rm e}^{-2K t} P_t |\nabla u|^2-|\nabla P_t u|^2\right)\,{\rm d}\mf_2
\\ \qquad
{}\geq \frac{1-{\rm e}^{-2 K t}}{K(N-n)}\left(\int_{X_2}\phi (\Delta P_{t} u)^2\,{\rm d}\mf_2
-\frac{2n}{N}\int_{X_2} \phi P_t\left( (\Delta u-2(N-n)u)^2\right)\,{\rm d}\mf_2 \right) .
\end{gather*}
This implies the claim.
\end{proof}

Let us show that $(X_2,d_2,\mf_2)$ satisfies the $\BL(K,N-n)$ condition.
Take $u\in W^{1,2}(X_2)$ and fix $s>0$.
Define
\[
v:=\Delta P_s u-2(N-n)P_s u=P_{s/2}\Delta P_{s/2}u-2(N-n)P_s u\in \TestF(X_2)
\]
and $v_x:= v-v(x)$ for each $x\in X_2$.
Then, the functions $v_x\colon X_2\to \R$ and
\[
[0,\infty)\times X_2\to \R,\qquad
(t,y)\mapsto P_t(v_x^2)(y)
\]
are continuous.
Thus, for fixed $\epsilon>0$ and any $x\in X_2$, there exists $\delta_x,\tau_x>0$ such that
we have
\[
|P_t(v_x^2)(y)|<\epsilon
\]
for any $y\in B_{\delta_x}(x)$ and $t\in (0,\tau_x)$.
Since $X_2$ is compact, there exist points $x_1,\ldots, x_k\in X_2$ ($k\in \Z_{>0}$) such that
\[
X_2=\bigcup_{i}B_{\delta_{x_i}} (x_i).
\]
Put $\tau:=\min{\tau_{x_i}}$.
Define
\[
\tilde{v}_i:= P_s u-(P_s u)(x_i)+\frac{1}{2(N-n)}\Delta( P_s u)(x_i).
\]
Then, we have $\tilde{v}_i\in \D(\Delta_{X_2})$ and
\[
\Delta \tilde{v}_i-2(N-n)\tilde{v}_i=\Delta P_s u-2(N-n)P_s u+2(N-n)(P_s u)(x_i)-\Delta (P_s u)(x_i)=v_{x_i}.
\]
Applying Claim~\ref{blmod} to $\tilde{v}_i$, for each $i$ and $t\in (0,\tau)$, we get
\begin{align*}
{\rm e}^{-2K t}P_t\big(|\nabla P_s u|^2\big)&\geq |\nabla P_{s+t} u|^2+\frac{1-{\rm e}^{-2Kt}}{K(N-n)}\bigg( (\Delta P_{s+t} u)^2-\frac{2n}{N}P_t\big(v_{x_i}^2\big) \bigg)\\
&\geq |\nabla P_{s+t} u|^2+\frac{1-{\rm e}^{-2Kt}}{K(N-n)}\bigg( (\Delta P_{s+t} u)^2-\frac{2n}{N}\epsilon\bigg)
\end{align*}
$\mf_2$-a.e.\ in $B_{\delta_{x_i}}(x_i)$.
Thus, for each $i$ and $t\in (0,\tau)$, we get
\[
{\rm e}^{-2K t}P_t\big(|\nabla P_s u|^2\big)\geq|\nabla P_{s+t} u|^2+\frac{1-{\rm e}^{-2Kt}}{K(N-n)}\left((\Delta P_{s+t} u)^2-\frac{2n}{N}\epsilon\right)
\]
$\mf_2$-a.e.\ in $X_2$. Letting $\epsilon\to 0$, we get
\[
{\rm e}^{-2K t}P_t\big(|\nabla P_s u|^2\big)\geq|\nabla P_{s+t} u|^2+\frac{1-{\rm e}^{-2Kt}}{K(N-n)}(\Delta P_{s+t} u)^2
\]
$\mf_2$-a.e.\ in $X_2$.
Letting $s\to 0$, we get the following inequality as the limit in $L^1(X_2)$:
\[
{\rm e}^{-2K t}P_t\big(|\nabla u|^2\big)\geq|\nabla P_{t} u|^2+\frac{1-{\rm e}^{-2Kt}}{K(N-n)}(\Delta P_{t} u)^2
\]
$\mf_2$-a.e.\ in $X_2$. This is the $\BL(K,N-n)$ condition.
\end{proof}

By Corollary~\ref{ihsl} and Proposition~\ref{blN-n}, we get Proposition~\ref{KN-n}.
By Propositions~\ref{spms} and~\ref{KN-n}, we get Proposition~\ref{strX}.

\section{Proof of the Main Theorem}
In this section we complete the proof of our main theorem.
\begin{Thm}\label{n-p+2}
For integers $n\geq 5$ and $2\leq p < n/2$ and a positive real number $\epsilon>0$, there exists $\delta=\delta(n,p,\epsilon)>0$ such that the following property holds.
Let $(M,g)$ be an $n$-dimensional closed Riemannian manifold with $\Ric_g\geq (n-p-1)g$, and assume one of the following:
\begin{itemize}\itemsep=0pt
\item $\lambda_1(\Delta_{C,p})\leq \delta$ and $\lambda_{n-p+1}(g)\leq n-p+\delta$,
\item $\lambda_1(\Delta_{C,n-p})\leq \delta$ and $\lambda_{n-p}(g)\leq n-p+\delta$.
\end{itemize}
Then, we have
\begin{gather*}
\lambda_{n-p+2}(g)\geq \frac{p(n-p-1)}{p-1}-\epsilon.
\end{gather*}
\end{Thm}
By the Lichnerowicz estimate for the first eigenvalue of the Laplacian acting on functions for metric measure spaces satisfying the $\RCD(n-p-1,p)$ condition \cite[Theorem~4.22]{EKS}:
\[
\lambda_1\geq \frac{p(n-p-1)}{p-1},
\]
we get Theorem~\ref{n-p+2} similarly to Theorem~\ref{mthm} below.
Thus, we only give the proof of Theorem~\ref{mthm}.

The following theorem is the main result of this article.
\begin{Thm}\label{mthm}
For integers $n\geq 5$ and $2\leq p < n/2$ and a positive real number $\epsilon>0$, there exists $\delta=\delta(n,p,\epsilon)>0$ such that the following property holds.
Let $(M,g)$ be an $n$-dimensional closed Riemannian manifold with $\Ric_g\geq (n-p-1)g$
satisfying one of the following:
\begin{itemize}\itemsep=0pt
\item $\lambda_1(\Delta_{C,p})\leq \delta$, $\lambda_{n-p+1}(g)\leq n-p+\delta$ and $\lambda_{n+1}(g)\leq p(n-p-1)/(p-1)+\delta$,
\item $\lambda_1(\Delta_{C,n-p})\leq \delta$, $\lambda_{n-p}(g)\leq n-p+\delta$ and $\lambda_{n+1}(g)\leq p(n-p-1)/(p-1)+\delta$.
\end{itemize}
Then, we have
\[d_{{\rm GH}}\left(M,S^{n-p}(1)\times S^p\left(\sqrt{\frac{p-1}{n-p-1}}\right)\right)\leq \epsilon.\]
\end{Thm}
\begin{proof}
We show the theorem by a contradiction.
Suppose that the theorem does not hold.
Then, there exists a sequence of $n$-dimensional closed Riemannian manifolds $\{(M_i,g_i)\}_{i=1}^\infty$ with $\Ric_{g_i}\geq (n-p-1)g_i$ that does not converge to $S^{n-p}(1)\times S^p\big(\sqrt{(p-1)/(n-p-1)}\big)$ and that satisfies one of the following:
\begin{itemize}\itemsep=0pt
\item $\lim\limits_{i\to\infty}\lambda_{n-p+1}(g_i)=n-p$, $\lim\limits_{i\to\infty}\lambda_{n+1}(g_i)=p(n-p-1)/(p-1)$ and $\lim\limits_{i\to \infty}\lambda_1(\Delta_{C,p},g_i)=0$,
\item $\lim\limits_{i\to\infty}\lambda_{n-p}(g_i)=n-p$, $\lim\limits_{i\to\infty}\lambda_{n+1}(g_i)=p(n-p-1)/(p-1)$ and  $\lim\limits_{i\to \infty}\lambda_1(\Delta_{C,n-p},g_i)=0$.
\end{itemize}
Taking a subsequence, we have that Assumption~\ref{convX} holds by Propositions~\ref{prev} and~\ref{impr}, the~Gromov compactness theorem (see also \cite[Theorem~11.1.10]{Pe3}, \cite[Theorem~4.54]{Ai3}) and \cite[Theorems~1.6 and~1.10]{CC1}.
Then, there exists a Borel measure $\mf_X$ on X such that $\mf=\nHa^{n-p}\times \mf_X$ holds and $(X,\mf_X)$ satisfies the $\RCD(n-p-1,p)$ condition by Proposition~\ref{strX}.
By the spectral convergence theorem \cite[Theorem~7.9]{CC3} and Theorem~\ref{n-p+2}, we have
\[
\lambda_{n-p+2}\big(S^{n-p}(1)\times X, \mf\big)=\cdots=\lambda_{n+1}\big(S^{n-p}(1)\times X, \mf\big)= \frac{p(n-p-1)}{p-1}.
\]
Since the spectrum of the Laplacian on $(S^{n-p}(1)\times X,\mf)$ coincides with
\[
\big\{\lambda_i\big(S^{n-p}(1),\nHa^{n-p}\big)+\lambda_j(X,\mf_X)\colon i,j\in\mathbb{Z}_{\geq 0}\big\}
\]
and $\lambda_{n-p+2}\big(S^{n-p}(1),\nHa^{n-p}\big)=2(n-p+1)>p(n-p-1)/(p-1)$,
we get that
\[
\lambda_1(X,\mf_X)=\cdots =\lambda_p(X,\mf_X)=\frac{p(n-p-1)}{p-1}.
\]
By the Obata Rigidity theorem for metric measure spaces satisfying the $\RCD$ condition \cite[Theorem~1.4]{Ket2} with scaling, we have that $(X,\mf_X)$ is isomorphic to either $\big(S^p(r_{n,p}),\nHa^{p}\big)$ or $\big(S^p_+(r_{n,p}),\nHa^{p}\big)$, where
$r_{n,p}:= \sqrt{(p-1)/(n-p-1)}$ and
 $S^p_+(r_{n,p})$ denotes the $p$-dimensional hemisphere with radius $r_{n,p}$.
In particular, $\big\{\big(M_i,g_i, \nHa^n\big)\big\}$ is a non-collapsing sequence.
Thus, we get $(X,\mf_X)$ is isomorphic to $\big(S^p(r_{n,p}),\nHa^{p}\big)$ by \cite[Theorem~6.2]{CC1}.
This contradicts to the assumption, and so we get the theorem.
\end{proof}

\subsection*{Acknowledgements}
I am grateful to Professor Shouhei Honda for helpful discussions.
I also thank Professor Dario Trevisan for answering my questions about the regular Lagrangian flow.
I wish to thank the \mbox{referees} for careful reading of the paper and making valuable suggestions.
This work was supported by RIKEN Special Postdoctoral Researcher Program.

\pdfbookmark[1]{References}{ref}
\LastPageEnding

\end{document}